\documentclass[11pt]{article}

\usepackage[T1]{fontenc}
\usepackage{microtype}

\usepackage[sc]{mathpazo}
\usepackage[sc]{xcolor}

\usepackage{dsfont}

\usepackage{environ}

\usepackage{amsthm,amsfonts,amsmath,amssymb,epsfig,color,float,graphicx,verbatim, enumitem}
\usepackage{multirow}

\usepackage{algorithm}
\usepackage{subcaption}
\usepackage[noend]{algpseudocode}
\makeatletter
\def\BState{\State\hskip-\ALG@thistlm}
\makeatother

\usepackage{threeparttable}

\usepackage{hyperref}
\hypersetup{
  colorlinks = true,
  urlcolor = {blueGrotto},
  linkcolor = {royalBlue},
citecolor = {limeGreen}
}
\definecolor{cpurple}{rgb}{0.6,0,0.6}
\definecolor{niceRed}{RGB}{190,38,38}
\definecolor{blueGrotto}{HTML}{059DC0}
\definecolor{royalBlue}{HTML}{057DCD}
\definecolor{navyBlueP}{HTML}{0B579C}
\definecolor{limeGreen}{HTML}{81B622}

\usepackage[nameinlink]{cleveref}
\crefname{ineq}{Inequality}{Inequality}
\creflabelformat{ineq}{#2{\upshape(#1)}#3}
\crefname{sub}{Subsection}{Subsection}
\creflabelformat{Subsection}{#2{\upshape(#1)}#3}
\crefname{sdp}{SDP}{SDP}
\creflabelformat{sdp}{#2{\upshape(#1)}#3}
\crefname{lp}{LP}{LP}
\creflabelformat{lp}{#2{\upshape(#1)}#3}

\usepackage{enumitem}

\theoremstyle{plain}
\newtheorem{theorem}{Theorem}[section]
\newtheorem{informal}{Informal Theorem}[]

\newtheorem{lemma}[theorem]{Lemma}
\newtheorem{informal theorem}[theorem]{Theorem (informal statement)}

\newtheorem{claim}[theorem]{Claim}

\theoremstyle{definition}
\newtheorem{definition}[theorem]{Definition}

\theoremstyle{remark}           \newtheorem{remark}[theorem]{Remark}
\newtheorem{example}[theorem]{Example}

\numberwithin{equation}{section}

\newcommand{\repeattheorem}[1]{\begingroup
  \renewcommand{\thetheorem}{\ref{#1}}\expandafter\expandafter\expandafter\theorem
  \csname reptheorem@#1\endcsname
  \endtheorem
  \endgroup
}

\newcommand{\repeatlemma}[1]{\begingroup
  \renewcommand{\thelemma}{\ref{#1}}\expandafter\expandafter\expandafter\lemma
  \csname replemma@#1\endcsname
  \endlemma
  \endgroup
}

\NewEnviron{replemma}[1]{\global\expandafter\xdef\csname replemma@#1\endcsname{\unexpanded\expandafter{\BODY}}\expandafter\lemma\BODY\unskip\label{#1}\endlemma
}

\NewEnviron{reptheorem}[1]{\global\expandafter\xdef\csname reptheorem@#1\endcsname{\unexpanded\expandafter{\BODY}}\expandafter\theorem\BODY\unskip\label{#1}\endtheorem
}

\newcommand{\vol}{\mathrm{vol}}
\newcommand{\lp}{\left}
\newcommand{\rp}{\right}
\newcommand\norm[1]{\left\| #1 \right\|}
\newcommand\snorm[2]{\left\| #2 \right\|_{#1}}
\newcommand\abs[1]{\left| #1 \right|}
\renewcommand\vec[1]{\boldsymbol{#1}}

\DeclareMathOperator*{\E}{\mathbf{E}}

\newcommand{\me}{\mathrm{e}}
\newcommand{\mrm}{\mathrm}
\def\d{\mathrm{d}}

\newcommand{\LL}{\mathcal{L}}

\newcommand{\mcal}{\mathcal}

\DeclareMathOperator*{\argmin}{argmin}
\DeclareMathOperator*{\argmax}{argmax}

\newcommand{\bx}{\boldsymbol{x}}
\newcommand{\by}{\boldsymbol{y}}
\newcommand{\bv}{\boldsymbol{v}}
\newcommand{\bu}{\boldsymbol{u}}
\newcommand{\bw}{\boldsymbol{w}}

\newcommand{\R}{\mathbb{R}}

\newcommand{\N}{\mathbb{N}}
\newcommand{\eps}{\epsilon}
\newcommand{\dtv}{\mathrm{d}_{\mathrm{TV}}}
\newcommand{\dkl}{\mathrm{D}_{\mathrm{KL}}}

\newcommand{\poly}{\mathrm{poly}}

\newcommand{\D}{\mathcal{P}}
\newcommand{\Ind}{\mathds{1}}
\newcommand{\1}{\Ind}

\newcommand{\wt}{\widetilde}

\usepackage{pgfplots}
\pgfplotsset{compat=1.7}

\newcommand{\diam}{\mathrm{diam}}
\newcommand{\dr}{\mathbf{D}}
\renewcommand{\bar}[1]{\overline{#1}}
\newcommand{\calL}{\mathcal{L}}

\title{A Sample-Based Taylor’s Theorem for Density Estimation}
\title{Taylor from Samples for Density Estimation}
\title{A Statistical Taylor Theorem\\ and Extrapolation of Truncated Densities}

\author{
\textbf{Constantinos Daskalakis} \\
Massachusetts Institute of Technology\\
\url{costis@csail.mit.edu}\\
\and
\textbf{Vasilis Kontonis} \\
University of Wisconsin-Madison\\
\url{kontonis@wisc.edu }\\
\and
\textbf{Christos Tzamos}\\
University of Wisconsin-Madison\\
\url{tzamos@wisc.edu}
\and
\textbf{Manolis Zampetakis}\\
Massachusetts Institute of Technology\\
\url{mzampet@mit.edu}\\
}

\date{}

\begin{document}

\maketitle

\begin{abstract}
We show a statistical version of Taylor's theorem and apply this result to non-parametric density estimation from truncated samples, which is a classical challenge in Statistics \cite{woodroofe1985estimating, stute1993almost}. The single-dimensional version of our theorem has the following implication: ``For any distribution $P$ on $[0, 1]$ with a smooth log-density function, given samples from the conditional distribution of $P$ on 
$[a, a + \varepsilon] \subset [0, 1]$, we can efficiently identify an approximation to  $P$ over the \emph{whole} interval $[0, 1]$, with quality of approximation that improves with the smoothness of $P$.''

To the best of knowledge, our result is the first in the area of non-parametric density estimation from truncated samples, which works under the hard truncation model, where the samples outside some survival set $S$ are never  observed, and applies to multiple dimensions. In contrast, previous works assume single dimensional data where each sample has a different survival set $S$ so that samples from the whole support will ultimately be collected.
\footnote{Accepted for presentation at the Conference on Learning Theory (COLT) 2021.}
\end{abstract}

\newpage

\section{Introduction} \label{sec:intro}

Non-parametric density estimation is a well-developed field in Statistics and 
Machine Learning \cite{wasserman2006all,tsybakov2008introduction} with 
applications to many scientific areas including economics 
\cite{ahamada2010non, li2007nonparametric}, and survival analysis
\cite{woodroofe1985estimating}. A central challenge in this field is estimating
a probability density function $\D(\vec{x})$ from samples, without making strong
parametric assumptions about the density. Of course, this is quite challenging
as $\D$ may exhibit very rich behavior which might be difficult or information
theoretically impossible to discern given a finite number of samples. Thus, to
make the task feasible at all, some constraints are placed on $\D$, typically in
the form of smoothness, which allows estimators to {\em interpolate} among the
observed samples. Indeed, a prominent method for non-parametric density
estimation is based on kernels 
\cite{WandJ1994, kernelmethods, Simonoff12, Scott15}, whose usual interpolating
estimate takes the form 
$\hat{\D}(\vec{x})={\frac{1}{n}} \sum_{i=1}^n k(\vec{x}_i; \vec{x})$, for some
kernel function $k(\cdot; \cdot)$, where $\vec{x}_1, \ldots, \vec{x}_n$ are the
observations from $\D$. In some settings it is also preferable to use kernels to 
estimate the log-density function \cite{canu2006kernel}. Even with smoothness
assumptions, the problem is challenging enough information theoretically, that
the achievable error takes the form $O(n^{-r/(r + d)})$, under various norms,
where $d$ is the dimension and $r$ is the assumed order of smoothness of $\D$
\cite{mcdonald2017minimax, li2007nonparametric, barron1992distribution}. 

  Despite the fact that both kernel based and histogram based estimators achieve
the optimal consistency rates, their estimators do not have a form that is
appealing for other statistical uses. For example, if our goal is to do inference
as well then it would be helpful if the estimated distribution is represented as
a member of an exponential family \cite{neyman1937smooth, good1963maximum, sriperumbudur2017density}. A
parallel line of research has hence devoted in 
\textit{exponential series estimators} of non-parametric densities, starting with
the celebrated work of Barron and Sheu \cite{BS91} which was later extended to
multidimensional settings by \cite{wu2010exponential}. Our work follows this
line of research and the estimators that we compute are always members of some
exponential family.

  The goal of this paper is to extend this literature from the traditional
{\em interpolating} regime to the much more challenging {\em extrapolating}
regime. In particular, we consider settings wherein we are constrained to
observe samples of $\D$ in a {\em subset} of its support, yet we want to procure
estimates $\hat{\D}$ that approximate $\D$ over its {\em entire} support. This 
question problem is motivated by truncated statistics, another well-developed
field in Statistics and Econometrics 
\cite{cohen1991truncated,heckman1976common,maddala1987limited,borsch1993smooth},
which targets statistical estimation in settings where the samples are truncated
depending on their membership in some set. Truncation may occur for several
reasons, ranging from measurement device saturation effects to data collection
practices, bad experimental design, ethical or privacy considerations that 
disallow the use of some data, etc. 

  Non-parametric density estimation from truncated samples is well-studied 
problem in statistics with many applications in economics and survival analysis
\cite{padgett1984nonparametric, woodroofe1985estimating, lai1991estimating, stute1993almost, gajek1988minimax, lai1991estimating}. 
However, due to the very challenging nature of this problem, all the previous
works on this topic consider only a soft truncation model that does not 
completely hide some part of the support but only decreases the probability of 
observing samples that lie in the truncation set. In particular, each sample 
$\vec{x}_i$ from $\D$ also samples a truncation set $S_i$ which then determines 
whether this sample is truncated or not. As a result, samples from the entire 
support are ultimately collected, thus the unknown density can be interpolated,
with some appropriate re-weighting, from those samples covering the entire
support. Moreover, the existing work only targets single-dimensional
densities despite the importance of non-parametric estimation in multiple
dimensions as we discussed above.  

  In this work, our goal is to solve the seemingly impossible problem of 
estimating a non-parametric density, even in parts of the support where we do 
not observe any sample! More precisely, we consider the more standard, in truncated
statistics, hard truncation model, wherein there is a fixed set $S$ that
determines whether a sample from $\D$ is truncated. We solve this
problem under slightly stronger, but similar, assumptions to the ones used in the
vanilla non-parametric density estimation problems. At the same time, we extend
the non-parametric density estimation from truncated samples to multi-dimensional settings, 
which is a significant generalization of the
existing work.

  Our main theorems, summarized in \Cref{sec:results}, can be interpreted
as a statistical version of Taylor's theorem, which allows us to use 
{\em truncated samples} from some sufficiently smooth density $\D$ and 
{\em extrapolate} from these samples an estimate $\hat{\D}$ which approximates
$\D$ on its {\em entire support}. The statistical rates achieved by our theorems
are slightly worse but comparable to those known in non-parametric density
estimation under {\em untruncated} samples, i.e., in the {\em interpolating} 
regime. It is an interesting open problem whether we can improve the novel
\textit{extrapolation} rates that we provide in this work, to match exactly the 
\textit{interpolation} rates of the vanilla non-parametric density estimation.

  From a technical point of view, a central challenge that we face is to bound 
the extrapolation error of multivariate polynomial approximation, which is a 
challenging problem and is a subject of active area of research. Our main 
technical contribution is to show a novel way to prove strong bounds on the
extrapolation error of our algorithms invoking only well-studied
\textit{anti-concentration} theorems, which is of independent interest and we 
believe that it will have applications beyond truncated statistics. More 
precisely, one of our main technical results is a ``Distortion of Conditioning''
lemma (\Cref{lem:distance_reduction}), providing a tight relationship
between the $\ell_1$ distance between two exponential families as computed under
conditioning on different subsets of the support. As we said, this lemma is 
proven using probabilistic techniques, and provides a viable route to prove our
statistical Taylor result in high dimensions, where polynomial approximation
theory techniques do not appear sufficient.
\medskip

\noindent \textbf{Further Related Work.}
  The use of exponential series estimators in non-parametric density estimation 
problems has many applications in other important problems in statistics, e.g., 
entropy estimation problems \cite{behmardi2011entropy, wang2013consistency}, 
estimation of copula functions \cite{chen2006efficient}, and online density 
estimation \cite{gokcesu2017online}. We believe that our results and tools can be
a cornerstone in extending these results to their truncated statistics
counterparts.

  Despite its long history, the field of truncated statistics suffers from a
lack of computationally and statistically efficient estimators in 
high-dimensions. Recent work, has made progress towards rectifying these 
limitations in parametric settings such as multi-variate normal estimation, 
linear regression, logistic regression, and support estimation
\cite{DaskalakisGTZ18,DaskalakisGTZ19,KontonisTZ19,IlyasZD20,CannoneDeS19}. 
Roughly speaking, this recent progress exploits the strong parametric 
assumptions about the density that is being estimated to extrapolate the density
outside of the truncation set. In comparison to this work, our goal here is to 
remove these parametric assumptions, allowing a very broad family of 
distributions to be extrapolated from truncated samples.

\subsection{Our Results and Techniques}
\label{sec:results}

  In this work we provide provable extrapolation of
non-parametric density functions from samples, i.e., given samples from the
conditional density on some subset $S$ of the support, we recover the
shape of the density function \emph{outside} of $S$. We consider densities
proportional to $e^{f(x)}$, where $f$ is a sufficiently smooth function. Our 
observation consists of samples from a density proportional to 
$\1_S(x) e^{f(x)}$, where $S$ is a known (via a membership oracle) subset of the
support. For this problem to even be well-posed we need further assumptions on
the density function. Even if we are given the exact conditional density 
$\1_S(x) e^{f(x)}$, it is easy to see that, if $f \notin C_{\infty}$, i.e., 
if $f$ is not infinitely times differentiable everywhere in the whole support,
there is no hope to extrapolate its curve outside of $S$; for a simple example,
if we observe a density proportional to $e^{|x|}$ truncated in $(-\infty,0]$ we
cannot extrapolate this density to $(0,+\infty)$, because we cannot distinguish
whether we are observing truncated samples from $e^{-x}$ or $e^{|x|}$. On the
other hand, if the log-density $f$ is analytic and sufficiently smooth, then the
value of $f$ at every $x$ can be determined only from local information, namely
its derivatives at a single point. This well known property of analytic functions
is quantified by Taylor's remainder theorem.  In this work we build on this and show 
that, even given  {\em samples} from a sufficiently smooth density and
even if these samples are {\em conditioned in a small subset of the support}, we
can still determine the function in the entire support and most importantly this 
can be done in a statistically and computationally efficient way.

  Since it is impossible to extrapolate non-smooth densities, we restrict our attention to smooth
functions $f$.  In particular, we assume that the $r$-th order derivatives of $f$
increase at most exponentially in $r$, i.e., 
$|f^{(r)}(\vec{x})| \leq M^r$ for some $M \in \R_+$ and all $\vec{x}$
(see \Cref{def:bounded and high order smooth}). Notice that similar
assumptions are standard in the interpolation regime of non-parametric density 
estimation \cite{BS91, wu2010exponential}.

  We start our exposition with the single-dimensional version of our 
extrapolation problem in \Cref{sec:one-dimensional}. 
In this setting it is easier to compare with the existing line of work
on non-parametric density estimation both in the vanilla non-truncated and in the
truncated setting. Moreover, in the single-dimensional setting, we are able to show a 
slightly stronger information theoretic result.
We assume that there exists some
unknown log-density function $f$, a known set $S$, and we observe samples from the
distribution $\D(f, S)$, which has density proportional to $\1_S(x) e^{f(x)}$. Our 
goal is to estimate the whole distribution $\D(f)$. For simplicity we assume
that $f$ is supported on $[0, 1]$ and hence $S \subseteq [0, 1]$. 
Our first step
is to consider the \emph{semi-parametric} class of densities $p$ that consists of 
polynomial series that can approximate the unknown non-parametric log-density $f$.
Then we truncate this polynomial series and we only consider densities of the form
$e^{p(x)}$, where $p$ is a degree $k$ polynomial, with large enough $k$; observe 
that these densities belong to an exponential family. 

  Our first result is that the polynomial which maximizes the likelihood with
respect to the \emph{conditional} distribution $\D(f, S)$ (we call this
polynomial the ``MLE polynomial'') approximates the density $e^{f(x)}$
\emph{everywhere} on $[0, 1]$, i.e. the MLE polynomial has small extrapolation 
error.   
\begin{definition}[MLE polynomial]
For some log-density $f: \R \mapsto \R$ denote $\D(f,S)$ the truncated distribution on $S$ 
with density function $\D(f, S;x) = \1_S(x) e^{f(x)}/\int_S e^{f(x)} d x $. 
We define the MLE polynomial $p^\ast$ of degree $k$ with respect to $\D(f,S)$ as
\[
p^\ast = \argmax_{p: \deg(p) \leq k} \E_{x \sim \D(f,S)}[ \log( \D(p, S; x) )] \,.
\]
\end{definition}

The extrapolation guarantee for the MLE polynomial cannot follow from the fact that, for 
example, the Taylor polynomial extrapolates, because the MLE polynomial and the
Taylor polynomial are in principle very different. 
It is not hard to argue that the MLE polynomial of sufficiently large degree has small interpolation 
error: it approximates well the density inside $S$.  In our result we 
show that the same polynomial has small extrapolation error and hence approximates the
density on the entire interval $[0, 1]$.

\begin{informal}[MLE Extrapolation Error, \Cref{thm:one_dimensional_exact_extrapolation}]
  \label{thm:informal_one_dimensional_exact_extrapolation} 
    Let $\D(f, [0, 1])$ be a probability distribution with sufficiently smooth
  log-density $f$ and let $\D(f, S)$ be its conditional distribution on 
  $S \subset [0, 1]$. The MLE w.r.t $\D(f, S)$ polynomial $p^*$ of degree
  $O(\log(1/\eps))$ satisfies $\dtv(\D(f, [0,1]), \D(p^*, [0,1])) \leq \eps$.
\end{informal}

Extending this result to multivariate densities is significantly more
challenging.  The reason is that multivariate polynomial interpolation is
much more intricate and is a subject of active research, see for example the
survey \cite{GS00}.  Instead of trying to characterize the properties of the
exact MLE polynomial we give an alternative method for obtaining multivariate
extrapolation guarantees that does not rely on multivariate polynomial
interpolation. Our approach uses the assumption that the set $S$ has non-trivial
volume, i.e., $\vol(S) \geq \alpha$ for some $\alpha > 0$. Observe that this
assumption is not needed in the single dimensional sample complexity analysis; 
in the multi-dimensional setting we need this assumption even to analyze the 
population model.

\begin{informal}[Multivariate MLE Extrapolation Error,
  \Cref{thm:low_dimensional_information_theoretic}]
  \label{thm:informal_multi_dimensional_exact_extrapolation} 
    Let $\D(f,[0,1]^d)$ be a probability distribution with sufficiently smooth
  log-density $f$ and let $\D(f, S)$ be its conditional distribution on 
  $S \subset [0,1]^d$ with $\vol(S) \geq \alpha$.  The MLE w.r.t $\D(f, S)$ 
  polynomial $p^*$ of degree $O(d^3/\alpha^2 + \log(1/\eps))$ satisfies
  $\dtv(\D(f, [0,1]^d), \D(p^*, [0,1]^d)) \leq \eps$.
\end{informal}

\noindent To prove 
\Cref{thm:informal_multi_dimensional_exact_extrapolation} we use a structural
result that quantifies the distortion of the metric space of exponential families
under conditioning. Given a polynomial $p$ with corresponding density 
$\D(p, [0, 1]^d)$ we consider the conditioning map that maps $\D(p, [0, 1]^d)$ to
the distribution $\D(p, S)$. We show that conditioning distorts the total
variation by a factor less than $(d/\alpha)^{O(k)}$, i.e., distributions that are
close in the image of the conditioning map are also close in the domain and vice
versa.

\begin{informal}[Distortion of Conditioning, \Cref{lem:distance_reduction}]
  \label{lem:informal_distance_reduction}
    Let $p, q$ be polynomials of degree at most $k$. For every 
  $S \subseteq [0, 1]^d$ with $\vol(S) \geq \alpha$ it holds
  \[ \left(d/\alpha\right)^{-O(k)} \leq
     \frac{\dtv(\, \D(p, [0,1]^d),\ \D(q, [0,1]^d)\, )}{\dtv(\D(p, S),\ \D(q, S) )} \leq 
     \left(d/\alpha\right)^{O(k)}. \]
\end{informal}

\noindent Using the above theorem, our strategy for showing  
\Cref{thm:informal_multi_dimensional_exact_extrapolation} is illustrated in 
\Cref{fig:distortion} and is as follows. Our first step is to use Taylor's
remainder theorem to prove that there exists a polynomial $p$, associated with 
$f$, such that both $\dtv(\D(p, S), \D(f, S))$ and 
$\dtv(\D(p, [0, 1]^d), $ $\D(f, [0, 1]^d))$ are very small when $p$ has 
sufficiently large degree. Next, we show that, by optimizing the likelihood function
on $S$ over the space of degree $k$ polynomials, we obtain the MLE polynomial $q$
which achieves very small total variation distance to $f$ on $S$, i.e. 
$\dtv(\D(q, S), \D(f, S))$ is also small. Hence, from the triangle inequality we
have that $\dtv(\D(q, S), \D(p, S))$ is also very small. The next step, which is
the crucial one, is that we can now apply our novel \Cref{lem:informal_distance_reduction} to obtain that 
$\dtv(\D(q, [0, 1]^d), \D(p, [0, 1]^d))$ blows up at most by a factor of
$(d/\alpha)^{O(k)}$. This argument leads to an upper bound on the extrapolation
error ($y$ in \Cref{fig:distortion}). The last key observation is that the 
quantity $\dtv(\D(p, S), \D(f, S))$ decreases faster than $(d/\alpha)^{-O(k)}$ as
the degree $k$ increases and hence we can make the extrapolation error
arbitrarily small by choosing sufficiently high degree.

\begin{figure}
  \centering
  \includegraphics[width=0.5\textwidth]{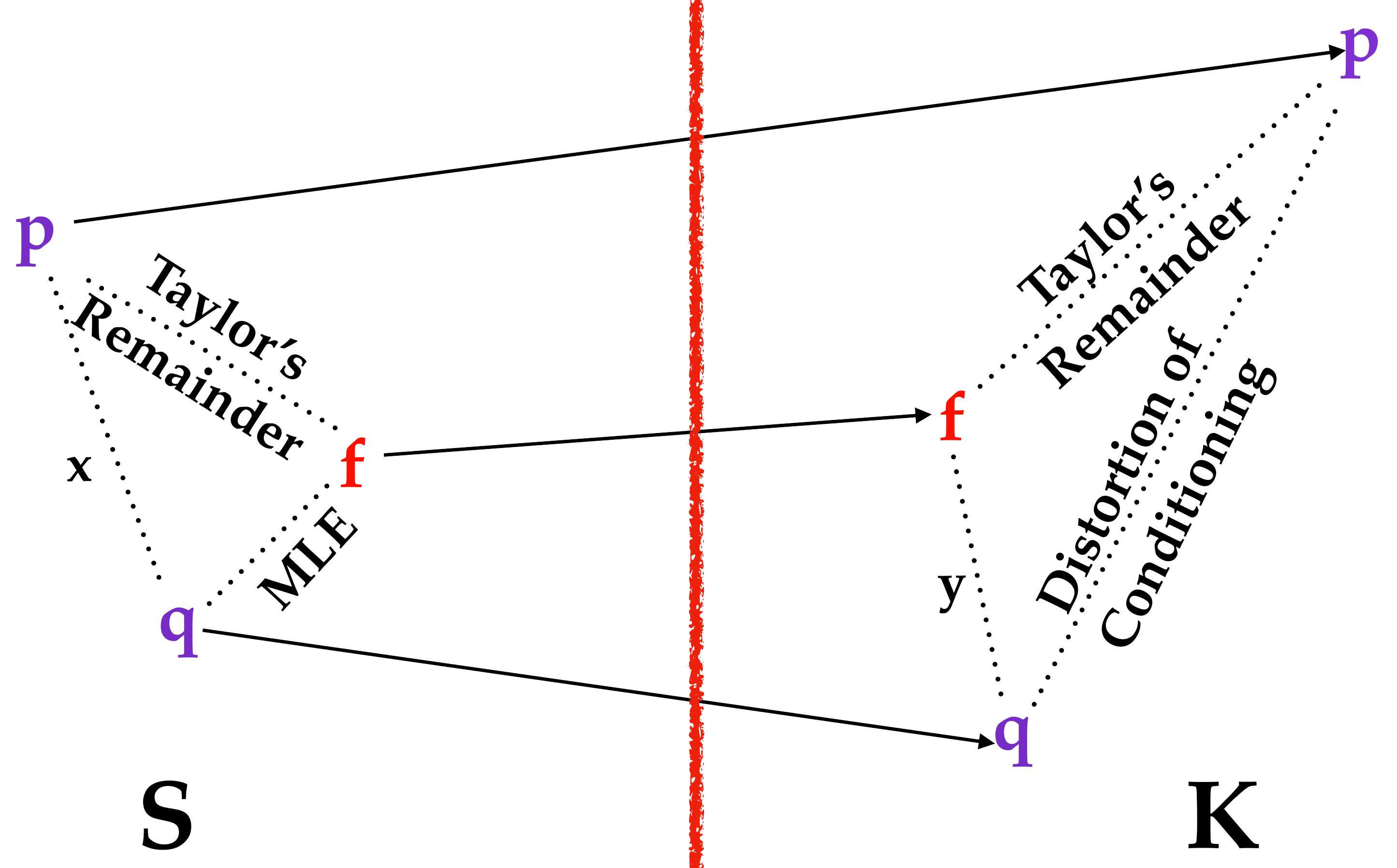}
  \caption{Using \Cref{lem:informal_distance_reduction} to show
  the extrapolation guarantees of MLE.  $K=[0,1]^d$. $p$ is the Taylor
  Polynomial of $f$: from Taylor's remainder theorem we know that, in both $S$
  and $K$, $p$ is very close to $f$. $q$ is the MLE polynomial on $S$: it is
  very close to $f$ in $S$. The distance $x$ is bounded by triangle inequality.
  The distance of $p$ and $q$ in $K$ is upper bounded by $x\ (d/\alpha)^{O(k)}$
  by ~\Cref{lem:informal_distance_reduction}. Finally, $y$ is the
  extrapolation error of the MLE polynomial $q$ on $K$ and is bounded by another
  triangle inequality. Overall, $y \leq \dtv(\D(f, K), \D(p,K)) +
  (d/\alpha)^{O(k)} x \leq \dtv(\D(f, K), \D(p,K)) + (d/\alpha)^{O(k)}
  (\dtv(\D(f, S), \D(p, S)) + \dtv(\D(f, S), \D(q, S)))$.}
  \label{fig:distortion}
\end{figure}

  So far we have argued about the extrapolation error of the population MLE
polynomial, i.e., we assume that we have access to the population distribution
$\D(f, S)$ and that we can maximize the population MLE with no error.  
Our next step is to show how we can incorporate the statistical error from the 
access to only finitely many samples from $\D(f, S)$ and to provide an efficient 
algorithm that computes the MLE polynomial with small enough approximation loss.

\begin{informal}
  [Extrapolation Algorithm: \Cref{thm:low_dimensional_algorithm}]
  Let $\D(f, [0,1]^d)$ be sufficiently smooth.  Let $S \subseteq [0,1]^d$ be
  such that $\vol(S) \geq \alpha$.  There exists an algorithm that draws 
  \[ N = 2^{\wt{O}(d^4/\alpha^2)} \cdot (1/\eps)^{O(d +\log(1/\alpha))} \] 
  samples from $\D(f, S)$, runs in time polynomial in the number of samples, and
  with probability at least $99\%$ finds a polynomial $q$ of degree
  $\wt{O}(d^3/\alpha^2) + O(\log(1/\eps))$ such that 
  $\dtv(\D(f, [0,1]^d), \D(q, [0,1]^d )) \leq \eps$.
\end{informal}
 It is well known that non-parametric density estimation (in the interpolation regime, i.e.~from untruncated samples) under
smoothness assumptions requires samples that depend exponentially in the
dimension, i.e.  the typical rate is $(1/\eps)^{\Theta(d)}$, see for example
\cite{tsy08}, \cite{macdon}.  The usual assumption is that the density has
bounded derivatives, i.e. it belongs to a Sobolev or Besov space.  Our
problem of extrapolating the density function is a strict generalization of
non-parametric density estimation and therefore our sample complexity
naturally scales as $(1/\eps)^{O(d + \log(1/\alpha))}$, where the
$\log(1/\alpha)$ reflects the impact of conditioning on a small volume set
$S$.  In particular, for sets of constant volume, in low dimensions,
we obtain a sample complexity that is polynomial in $1/\eps$.
\section{Definitions and Preliminaries} \label{sec:prelims}

\noindent \textbf{Notation.} Let $K \subseteq \R^d$ and $B \in \R_+$, we define
$L_{\infty}(K, B)$ to be the set of all functions $f : K \to \R$ such that
$\max_{\vec{x} \in K} \abs{f(\vec{x})} \le B$. We may use $L_{\infty}(B)$ instead
of $L_{\infty}([0, 1]^d, B)$. We also define 
$\diam_{p}(K) = \sup_{\vec{x}, \vec{y} \in K} \norm{\vec{x} - \vec{y}}_{p}$
where $\norm{\cdot}_p$ is the usual $\ell_p$-norm of vectors. Let
$\R^{d \times \cdots \text{($k$ times)} \cdots \times d}$ be the set of
$k$-order tensors of dimension $d$, which for simplicity we will denote by
$\R^{d^k}$. For $\vec{\alpha} \in \N^d$, we define the factorial of the
multi-index $\vec{\alpha}$ to be
$\vec{\alpha}! = (\alpha_1 !) \cdots (\alpha_d  !)$. Additionally for any
$\vec{x}, \vec{y}, \vec{z} \in \R^d$ we define
$\vec{z}^{\vec{\alpha}} = z_1^{\alpha_1} \cdots z_d^{\alpha_d}$ and in
particular
$(\vec{x} - \vec{y})^{\vec{\alpha}} = (x_1 - y_1)^{\alpha_1} \cdots (x_d - y_d)^{\alpha_d}$.

\begin{remark} \label{rem:domainGenerality}
    Throughout the paper, for simplicity of exposition, we will consider the
  support $K$ of the densities that we aim to learn to be the
  hypercube $[0, 1]^d$. Our results  hold for arbitrary convex sets with
  the following property $[a, b]^d \subseteq K \subseteq [c, d]^d$. Then all
  our results will be modified by multiplying with a function of
  $R \triangleq \frac{d - c}{b - a}$. 
\end{remark}

\subsection{Multivariate Polynomials} \label{sec:model:polynomials}

  When we use a polynomial to define a probability distribution, as we will see
in Section \ref{sec:model:distributions}, the constant term of the polynomial
has to be determined from the rest of the coefficients so that the resulting
function integrates to $1$. For this reason we can ignore the constant term. 
A polynomial $q : \R^d \to \R$ is a function of the form
\begin{align} \label{eq:polynomialGeneralSimple}
  q(\vec{x}) = \sum_{\vec{\alpha} \in \N^d, 0 < \abs{\vec{\alpha}} \le k} v_{\vec{\alpha}} \vec{x}^{\vec{\alpha}}
\end{align}
\noindent where $v_{\vec{\alpha}} \in \R$. The monomials of degree $\le k$
can be indexed by a multi-index $\vec{\alpha} \in \N^d$ with
$\abs{\vec{\alpha}} \le k$ and any polynomial belongs to the vector space
defined by the monomials as per \cref{eq:polynomialGeneralSimple}.

  To associate the space of polynomials with a Euclidean space we can use any
ordering of monomials, for example, the lexicographic ordering. Using this
ordering we can define the \textit{monomial profile of degree $k$},
$\vec{m}_k : \R^d \to \R^{t_k - 1}$, as
$\left(\vec{m}_k(\vec{x})\right)_{\vec{\alpha}} = \vec{x}^{\vec{\alpha}}$
where $t_k = \binom{d + k}{k}$ is equal to the number of
monomials with $d$ variables and degree at most $k$ and where we  abuse 
notation to index a coordinate in $\R^{t_k-1}$ via a multi-index
$\vec{\alpha} \in \N^d$ with $\abs{\vec{\alpha}} \le k$ and
$\vec{\alpha} \neq \vec{0}$; this can be formally done using the lexicographic
ordering. Therefore the vector space of polynomials is homomorphic to the
vector space $\R^{t_k-1}$ via the following correspondence
\begin{align} \label{eq:polynomialDefinitionVectorSpace}
  \vec{v} \in \R^{t_k - 1} \leftrightarrow \vec{v}^T \vec{m}_k(\vec{x}) \triangleq q_{\vec{v}}(\vec{x}).
\end{align}

We denote by $\mathcal{Q}_{d, k}$ the space of polynomials of degree at most
$k$ with $d$ variables and zero constant term, where we might drop $d$ from the
notation if it is clear from  context.

\subsection{High-order Derivatives and Taylor's Theorem}
\label{sec:model:derivatives}

  In this section we will define the basic concepts about high order
derivatives of a multivariate real-valued function $f : K \to \R$, where
$K \subseteq \R^d$.
\smallskip

  Fix $k \in \N$ and let $\vec{u} \in [d]^k$. We define the order $k$
derivative of $f$ with index $\vec{u} = (u_1, \dots, u_k)$ as
$\dr^k_{\vec{u}} f(\vec{x}) = \frac{\partial^k f}{\partial x_{u_1} \cdots \partial x_{u_k}}(\vec{x})$,
observe that $\dr^k_{\vec{u}} f$ is a function from $K$ to $\R$. The order $k$
derivative of $f$ at $\vec{x} \in S$ is then the tensor
$\dr^k f (\vec x) \in \R^{d^k}$ where the entry of $\dr^k f (\vec x)$ that
corresponds to the index $\vec{u} \in [d]^k$ is
$(\dr^k f (\vec x))_{\vec{u}} = \dr^k_{\vec{u}} f (\vec x)$. Due to symmetry the
$k$-th order derivatives can be indexed with a multi-index
$\vec{\alpha} \triangleq (\alpha_1, \dots, \alpha_d) \in \N^d$, with
$\abs{\vec{\alpha}} = \sum_{i = 1}^d \alpha_i = k$, as follows
$\dr_{\vec{\alpha}} f(\vec{x}) = \frac{\partial^{\abs{\vec{\alpha}}} f}{\partial^{\alpha_1} x_1 \cdots \partial^{\alpha_d} x_d}(\vec{x})$.
Observe that the $k$-order derivative of $f$ is a function
$\dr^k f : K \to \R^{d^k}$.
\medskip

\noindent \textbf{Norm of High-order Derivative.} There are several ways to
define the norm of the tensor and hence the norm of a $k$-order derivative of
a multi-variate function. Here we will define only the norm that we will use
in the rest of the paper as follows
\begin{align} \label{eq:derivativeNormMax}
  \norm{\dr^k f}_{\infty} \triangleq \sup_{\vec{x} \in K} \max_{\vec{u} \in [d]^k} \abs{\dr^k_{\vec{u}} f(\vec{x})} = \sup_{\vec{x} \in K} \max_{\vec{u} \in [d]^k} \abs{\frac{\partial^k f}{\partial x_{u_1} \cdots \partial x_{u_k}}(\vec{x})}.
\end{align}
\noindent Observe that this definition depends on $K$, but for ease of notation we
eliminate $K$ from the notation and we make sure that this set will be clear from
the context. For most part of the paper $K$ is the box $[0, 1]^d$. We next define
the $k$-order Taylor approximation of a multi-variate function.

\begin{definition}(Taylor Approximation) \label{def:TaylorApproximation}
    Let $f : K \to \R$ be a $(k + 1)$-times differentiable function on
  the convex set $K \subseteq \R^d$. Then we define
  $\bar{f}_k(\cdot; \vec{x})$ to be the $k$-order Taylor approximation of $f$
  around $\vec{x}$ as follows
  $\bar{f}_k(\vec{y}; \vec{x}) = \sum_{\vec{\alpha} \in \N^d, \abs{\vec{\alpha}} \le k} \frac{\dr_{\vec{\alpha}} f(\vec{x})}{\vec{\alpha}!} (\vec{y} - \vec{x})^{\vec{\alpha}}$.
\end{definition}

\noindent We are now ready to state the main application of Taylor's Theorem. 
For the proof of this theorem together with the statement of the multi-dimensional
Taylor's Theorem we refer to \Cref{sec:app:taylor}.

\begin{theorem}[Corollary of Taylor's Theorem] \label{thm:TaylorTheoremApp}
    Let $K \subseteq \R^d$ and $f : K \to \R$ be a $(k + 1)$-times
  differentiable function such that $\diam_{\infty}(K) \le R$ and
  $\norm{\dr^{k + 1} f}_{\infty} \le W$, then for any
  $\vec{x}, \vec{y} \in K$ it holds that
  $
    \abs{f(\vec{y}) - \bar{f}_k(\vec{y}; \vec{x})} \le \left(15 d/k \right)^{k + 1} \cdot R^{k + 1} \cdot W.
  $
\end{theorem}

\subsection{Bounded and High-order Smooth Functions} \label{sec:model:boundedHighOrderSmooth}

  In this section we define the set of functions that our statistical Taylor
theorem applies. It is also the domain of function with respect to which we are
solving the non-parametric truncated density estimation problem. This set of
functions is very similar to the functions consider for interpolation of
probability densities from exponential families \cite{BS91}. We note that in
this paper our goal is to solve a much more difficult problem since our goal is
to do extrapolation instead of interpolation. We call the set of function that
we consider \textit{bounded and high-order smooth functions}.

\begin{definition}[\emph{Bounded and High-order Smooth Functions}] \label{def:bounded and high order smooth}
    Let $K = [0, 1]^d$, we define the set $\calL(B, M)$ of functions
  $f : K \to \R$ for which the following conditions are satisfied.
  \begin{enumerate}
    \item[$\blacktriangleright$] \emph{(Bounded Value)} It holds that
    $\max_{\vec{x} \in K} \abs{f(\vec{x})} \le B$.
    \item[$\blacktriangleright$] \emph{(High-Order Smoothness)} For any
    natural number $k$ with $k \ge k_0$, $f$ is $k$-times continuously differentiable and it holds that
    $\norm{\dr^k f}_{\infty} \le M^k$.
  \end{enumerate}
  We note that the definition of the class $\calL$ depends on $k_0$ as well but
  for ease of notation we don't keep track of this dependence and we treat $k_0$
  as a constant throughout the paper.
\end{definition}

\subsection{Probability Distributions} \label{sec:model:distributions}

  We are now ready to define probability distributions with a given log-density
function.

\begin{definition}
    Let $S \subseteq \R^d$ and let $f : S \to \R$ such that
  $\int_S \exp(f(\vec{x})) d \vec{x} < \infty$. We define the distribution
  $\D(f, S)$ with log-density $f$ supported on $S$ to be the distribution with
  density
  \begin{align*}
    \D(f, S; \bx) = \frac{\1_S(\vec x)\ e^{f(\vec x)}}{\int_S e^{f(\vec x)}\ \d x} = \1_S(\vec x)\ \exp(f(\vec x) - \psi(f,S))\,, ~~~~\text{where}~~~ \psi(f, S) = \log \int_S e^{f(x)} \d \bx\,.
  \end{align*}
  \noindent If $f$ is equal to a $k$ degree polynomial
  $q_{\vec{v}} \in \mathcal{Q}_k$ with coefficients $\vec{v} \in \R^{t_k - 1}$, where $t_k = \binom{d + k}{k}$,
  then instead of $\D(q_{\vec{v}}, S)$ we write
  $\D(\vec{v}, S)$. Finally, let $T \subseteq S$, we define 
  $\D(f, S; T) = \int_T \D(f, S; \bx) \d \bx$.
\end{definition}

  Our main focus in this paper is on probability distributions
$\D(f, [0, 1]^d)$ where $f \in \calL(B, M)$ for some known parameters $B, M$.
More specifically our main goal is to approximation the density of
$\D(f, [0, 1]^d)$ given samples from $\D(f, S)$, where $S$ is a measurable
subset of $[0, 1]^d$.

\section{Single Dimensional Densities} \label{sec:one-dimensional}

  In this section we show our Statistical Taylor Theorem for single-dimensional
densities. We keep this analysis separate from our main multi-dimensional
theorem for several reasons. First, there exists a great body of work on 
single-dimensional non-parametric estimation problems in the vanilla setting and
more specifically in truncated estimation problems this is the only setting that
has been considered so far. Therefore, it is easier to compare the estimators and
results that we get with the existing results. In fact this is the strategy that
is followed in other multi-dimensional non-parametric estimation problems, e.g., 
see \cite{wu2010exponential}. Another reason is that in the single dimensional
setting we are able to obtain a slightly stronger information theoretic result
using more elementary tools, although the analysis of our efficient algorithmic
procedure is the same as in multiple dimensions. Finally, the single dimensional
setting serves as a nice example where the difference between interpolation and
extrapolation is more clear.

  In this section our goal is to estimate the density of the distribution 
$\D(f, [0, 1])$ using only samples from $\D(f, S)$, where the log-density $f$ is a
bounded and high-order smooth function, i.e. $f \in \calL(B, M)$, and $S$ is a
measurable subset of $[0, 1]$. As a first step we need to understand what is a 
sufficient degree for a polynomial to well-approximate
(\Cref{sec:one-dimensional:sufficientDegree}) this is the part that is 
different compared to the multi-dimensional case that we present in 
\Cref{sec:multivariate}. In  \Cref{sec:one-dimensional:optimizationError} we state the application of our
general multi-dimensional statistical and computational result to the single
dimensional case where the assumptions and guarantees have a simpler form.

\subsection{Identifying the Sufficient Degree} \label{sec:one-dimensional:sufficientDegree}

  In this section we assume population access to $\D(f, S)$ and our goal is to
identify a polynomial $q$ such that $\D(q, [0, 1])$ approximates $\D(f, [0, 1])$.
In particular, we want to answer the question: if $q$ minimizes the KL-divergence
between $\D(q, S)$ and $\D(f, S)$, what can we say about the total variation 
distance between $\D(q, [0, 1])$ and $\D(f,[0, 1])$? Moreover, how does this
depend on the degree of $q$? As the degree of $q$ grows, it certainly allows
$\D(q, S)$ to come closer to $\D(f, S)$. The natural thing to expect hence is that
the same is true of $\D(q, [0, 1])$, coming closer to $\D(f, [0,1])$.
Unfortunately, this is not necessarily the case, because it could be that, as the
degree of the polynomial increases, the approximant $\D(q, S)$ overfits to 
$\D(f, S)$, so extrapolating to $[0, 1]$ fails to give a good approximation to
$\D(f, [0, 1])$. This is the main technical difficulty of this section. 

  In the next theorem, we show is that if the function is high-order smooth, then
there is some threshold beyond which we get better approximations using higher
degrees, i.e. overfitting does not happen for any degree above some threshold. We
illustrate this behavior in \Cref{exm:running1D}.

\begin{theorem}[MLE Polynomial Extrapolation Error]
  \label{thm:one_dimensional_exact_extrapolation}
    Let $I = [0, 1] \subseteq \R$ and $f: I \mapsto \R$ be a function that is
  $(k+1)$-times continuously differentiable on $I$, with
  $\abs{f^{(k+1)}(x)} \leq M^{k+1}$ for all $x \in I$. Let $S \subseteq I$ be a
  measurable set such that $\vol(S) > 0$, and $p$ be a polynomial of degree at
  most $k$ defined as
  \( p = \argmin_{q \in \mathcal{Q}_k} \dkl(\D(f,S) \| \D(q, S))\,. \)
  Then, it holds that
  $\dkl( \D(f, I) \| \D(p, I))\leq e^{W_k} W_k^2$ where $W_k = \frac{M^{k+1}}{(k+1)!} \,. $
\end{theorem}

\noindent The proof of \Cref{thm:one_dimensional_exact_extrapolation} can
be found in  \Cref{sec:proof:one_dimensional_exact_extrapolation}. 
We also note that we can prove a more general version of  \Cref{thm:one_dimensional_exact_extrapolation} where $I$ is any interval $[a, b]$.
The difference in the guarantees is that the term $W_k$ will be multiplied by
$R^{k + 1}$ where $R \triangleq b - a$.
\smallskip

\noindent To convey the motivation for our theorem and illustrate its
guarantees, we use the following example.

\begin{example} \label{exm:running1D}
    Let $f(x) = \sin(10 \cdot x)$ and $S = [0, 1/2]$. Our goal is to
  estimate the probability distribution $\D(f, [0, 1])$ using only samples from
  $\D(f, S)$. The guarantees of 
  \Cref{thm:one_dimensional_exact_extrapolation} are illustrated in  
  \Cref{fig:example1D:degreeError} where we can see the density of the 
  distributions $\D(f, [0, 1])$, $\D(f, S)$, $\D(q, S)$ and $\D(q, [0, 1])$ for
  various values of the degree of $q$, where $q$ is always chosen to minimize the
  KL-divergence between $\D(f, S)$ and $\D(q, S)$, i.e., using no further
  information about $\D(f, [0, 1])$. As we see, $\D(q, S)$ approximates
  $\D(f, S)$ very well for all the presented degrees, with a marginal
  improvement as the degree of $q$ is increased.

An important observation is that the
  approximation error between $\D(f, [0, 1])$ and $\D(q, [0, 1])$ is not
  monotone in any range of degrees of $q$. In particular, when the degree of $q$ is $10$
  then $\D(q, [0, 1])$ is reasonably close to $\D(f, [0, 1])$ while when the
  (degree of $q$) $= 12$ then $\D(q, [0, 1])$ is way off. This suggests that overfitting occurs for degree $ = 12$. The point of 
  \Cref{thm:one_dimensional_exact_extrapolation} is that it guarantees that, for degree
  greater than a threshold, overfitting cannot happen and $\D(q, [0, 1])$ will
  always be a good approximation to $\D(f, [0, 1])$.

  \begin{figure}[t]
    \centering
    \begin{subfigure}[b]{0.31\textwidth}
      \includegraphics[width=\textwidth]{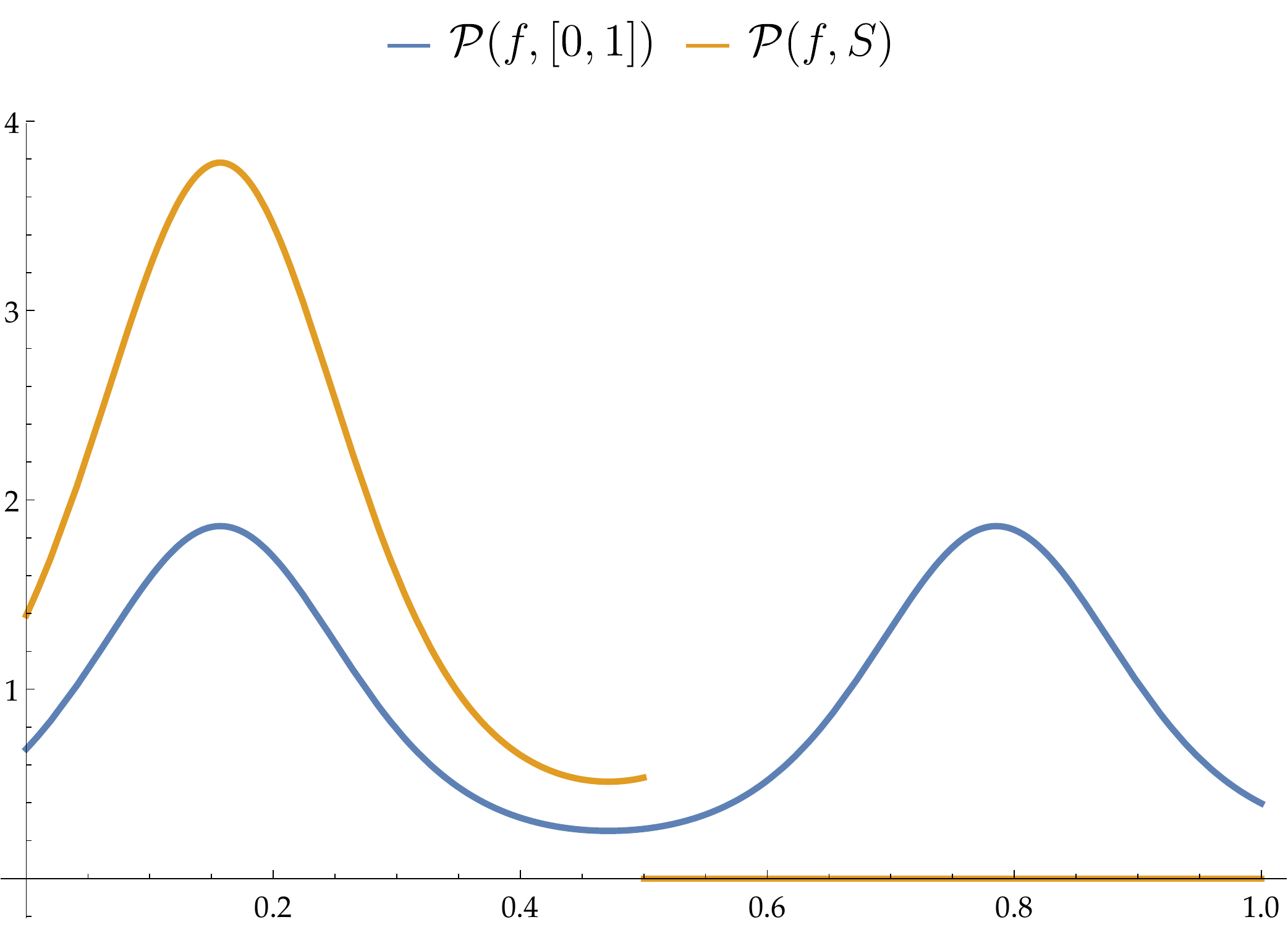}
      \caption{The densities of $\D(f, [0, 1])$ and $\D(f, S)$.}
    \end{subfigure} ~
    \begin{subfigure}[b]{0.31\textwidth}
      \includegraphics[width=\textwidth]{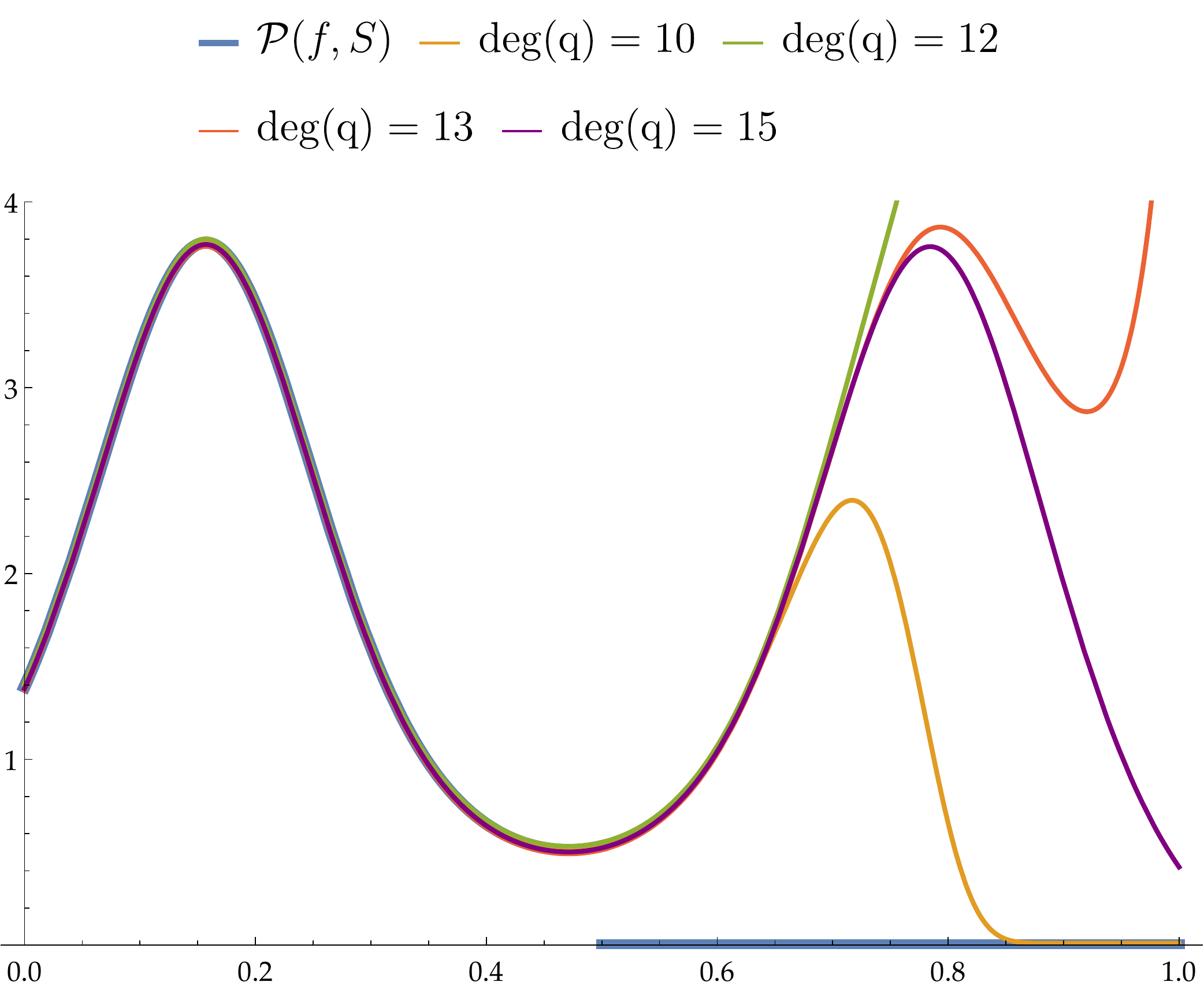}
      \caption{$\D(f, S)$ and normalized on $S$ functions $e^{q(\cdot)}$ for
      different $\mathrm{deg}(q)$.}
    \end{subfigure} ~
    \begin{subfigure}[b]{0.31\textwidth}
      \includegraphics[width=\textwidth]{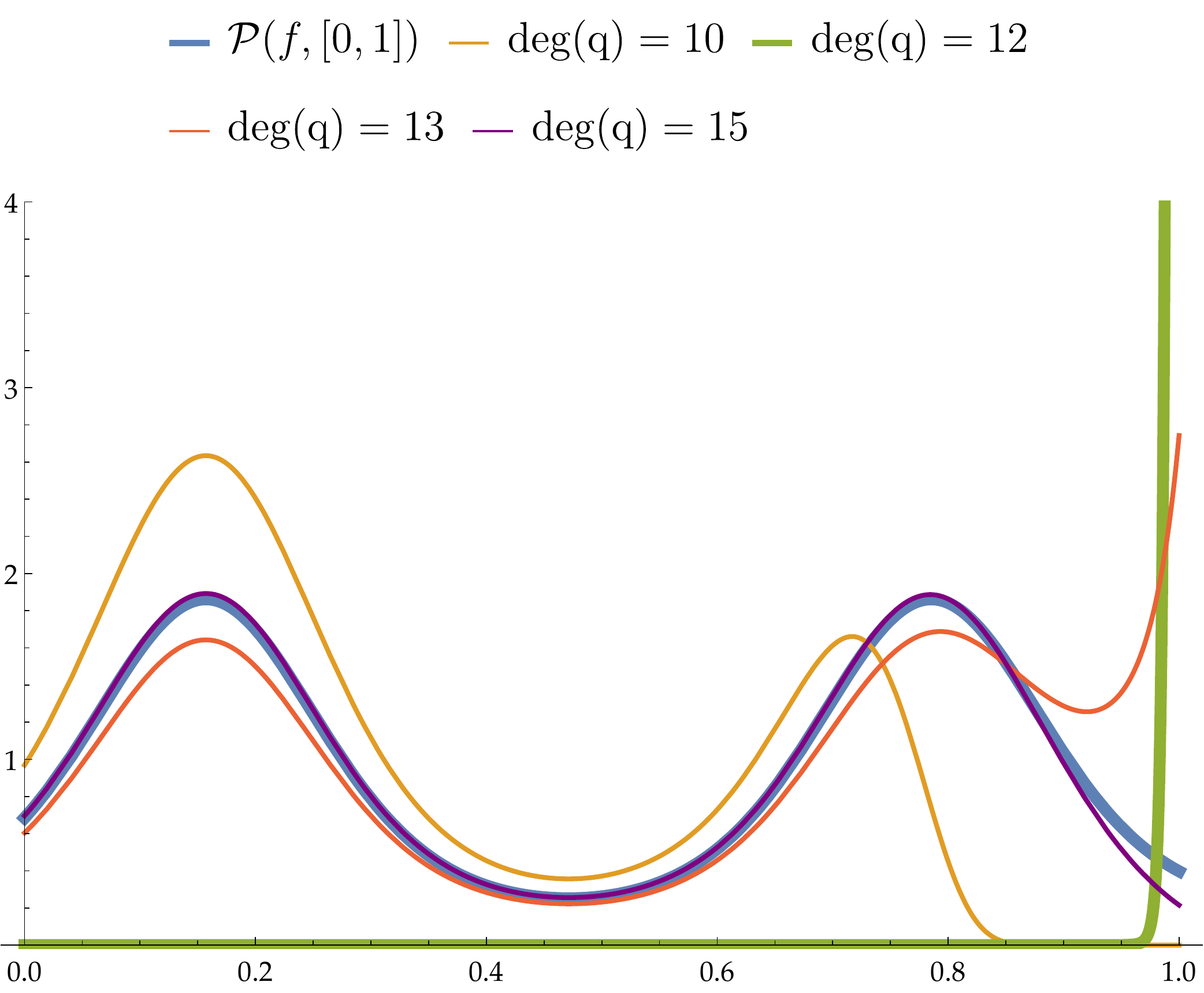}
      \caption{The densities of $\D(f, [0, 1])$ and $\D(q, [0, 1])$ for different $\mathrm{deg}(q)$.}
    \end{subfigure}

    \caption{In figure (a) we can see the probability density functions of the
    distributions $\D(f, [0, 1])$ and $\D(f, S)$. In figure (b) we have the
    density of $\D(f, S)$ together with the functions $\exp(q(x))$ for various
    degrees of $q$ normalized so that the integral on $S$ is 1. As we can see
    all the degrees approximate very well the conditional density but they have
    completely different behavior outside $S$. In figure (c) we can see the
    densities $\D(f, [0, 1])$ and $\D(q, [0, 1])$ for various degrees of $q$.
    The difference between (b) and (c) is that in (c) the functions are
    normalized so that their integral over $[0, 1]$ is equal to $1$ whereas in
    (b) the integral over $S$ is equal to $1$.}
    \label{fig:example1D:degreeError}
  \end{figure}
\end{example}

\subsection{Handling the Optimization Error} \label{sec:one-dimensional:optimizationError}

Our next theorem handles the approximation error that is
introduced, due to access to finite samples.

\begin{theorem}[Approximate MLE Polynomial Extrapolation Error]
  \label{thm:one_dimensional_aproximate_extrapolation}
    Let $I = [0,1]$, $f \in \LL(B, M)$ be a function supported on $I$,
  $S \subseteq I$ be a measurable set such that $\vol(S) \geq \alpha$,
  $D_k$ be the convex set
  $D_k =
  \{
    p\in\mathcal{Q}_k: p \in L_\infty(I, 3 B)
  \}\, 
  $, where $k = \Omega(M + \log(1/\eps))$, and let also
  $r^*_k = \min_{p \in D_k} \dkl(\D(f,S) \| \D(p, S))$. If some $q$ with
  $q \in D_k$ satisfies
  \begin{align} \label{eq:singleDimensional:approximateOptimizationRequirement}
    \dkl(\D(f, S) \| \D(q, S)) \leq r_k^* + 2^{- O( k \log(1/\alpha) + B)},
  \end{align}
  then it holds that $\dtv(\D(f, I), \D(q, I)) \leq \eps$.
\end{theorem}

  The proof of \Cref{thm:one_dimensional_aproximate_extrapolation} can be
found in  \Cref{sec:proof:one_dimensional_aproximate_extrapolation}.
Next we argue that we can efficiently compute a polynomial $q$ that satisfies
\Cref{eq:singleDimensional:approximateOptimizationRequirement}. Unfortunately, 
the proof of this step is not simplified in the 1-D case and we need to invoke
our general multi-dimensional theorem with the assumptions and guarantees
simplified due to the single dimensionality of the distribution. For more details
about the specific algorithm that we use we refer to \Cref{sec:sgd}. This
leads to the following theorem.

\begin{theorem}[1-D Statistical Taylor Theorem]
  \label{thm:1DStatisticalTaylor}
    Let $I = [0,1]$, $f \in \LL(B, M)$ be a function supported on $I$, and
  $S \subseteq I$ be a (measurable) subset of $I$ such that
  $\vol(S) \geq \alpha$. There exists an algorithm that draws
  $N = 2^{\wt{O}((M +\log(1/\eps)) \log(1/\alpha) + B)}$ samples from
  $\D(f, S)$, runs in time polynomial in the number of samples, and outputs a
  vector of coefficients $\vec v $ such that
  $\dtv(\D(f, K), \D(q_{\vec{v}}, K)) \leq \eps$.
\end{theorem}

\begin{proof}
We define the convex set
$D_k = \left\{ \vec v \in \R^m~:~ \max_{\vec{x} \in [0, 1]^d} \abs{\vec v^T \vec m_k(\bx)} \leq 3 B \right\}$,
where $m= \binom{d + k}{k}-1$. From
\Cref{thm:one_dimensional_aproximate_extrapolation} we can fix $k = O(M+ \log(1/\eps))$ and find a candidate
$\vec v$ such that
$\dkl(\D(f,S)\|\D(q_{\bv},S)) \leq \min_{\bu \in D_k}
\dkl(\D(f,S)\|\D(q_{\bu},S)) + 2^{-O(k \log(1/\alpha) + B)}$.
Therefore, from \Cref{thm:KL_minimization} we obtain that
with $N=2^{\wt{O}(k \log(1/\alpha) +B)}$ samples and
in time $\poly(N)$, we can compute such a candidate.
\end{proof}
 \section{Multi-Dimensional Densities} \label{sec:multivariate}

  In this section we show the general form of our Statistical Taylor Theorem
that applies to multi-dimensional densities. Although the techniques used in this section and involved and possibly of independent interest, our strategy to prove
this theorem is similar to the strategy that we followed in Section 
\ref{sec:one-dimensional}: (1) we identify the sufficient degree that we need to
use, (2) we handle approximation errors that we get as a result of finite sample,
and (3) we design an efficient algorithm to compute the solutions that are 
information-theoretically shown to exist. This leads to our main theorem.

\begin{theorem}[Multi-Dimensional Statistical Taylor Theorem] \label{thm:low_dimensional_algorithm}
    Let $f : [0, 1]^d \to \R$ with $f \in \LL(B, M)$ and $S \subseteq [0, 1]^d$
  such that $\vol(S) \geq \alpha$. Let
  $k = \wt{\Omega}(d^3M/\alpha^2 + B) + 2 \log(1/\eps)$, then there exists an
  algorithm that uses $N$ samples from $\D(f, S)$ with
  \[ N = 2^{\wt{O}(d^4 M /\alpha^2 + B d)} \cdot (1/\eps)^{O(d + \log(1/\alpha))}, \]
  runs in $\mathrm{poly}(N)$ time and outputs a vector
  $\vec{v}$ such that $\dtv(\D(f, K), \D(q_{\vec v}, K)) \leq \eps$.
\end{theorem}

\noindent The main bottleneck in the proof of the above theorem is that in the
multi-dimensional polynomial interpolation theory there are no sufficient tools
to prove the extrapolation properties of an estimator that can be computed 
efficiently. Our formulation of the extrapolation problem in the
language of density estimation enables us to use anti-concentration results
to prove extrapolation results for polynomial approximations. In particular, 
we prove our main lemma which we call ``Distortion of Conditioning Lemma'' which
we believe is of independent interest.

\subsection{Identifying the Sufficient Degree -- The Distortion of Conditioning Lemma} \label{sec:multiD:sufficientDegree}

The goal of this section is to identify the sufficient degree so that the MLE
polynomial approximates well the true density in the whole domain $K = [0, 1]^d$, 
i.e., it has small extrapolation error.

\begin{theorem}[Multivariate MLE Polynomial Extrapolation Error]
  \label{thm:low_dimensional_information_theoretic}
    Let $K = [0,1]^d \subseteq \R^d$, $f \in \LL(B, M)$ be function supported
  on $K$, and $S \subseteq K$ be a measurable subset of $K$ such that
  $\vol(S) \geq \alpha$.  Moreover, define
  \( k = \wt{\Omega}\left(\frac{d^3 M}{\alpha^2} + \log\left(\frac{2^B}{\eps}\right)\right) \),
  $
  D = \{ \vec v : \max_{\bx \in K} \abs{\vec v^T \vec m_k(\bx)} \le 3 B\}
  $,
  Moreover, let 
  $r^*_k = \min_{\vec u \in D} \dkl(\D(f, S) \| \D(\vec u, S))$.
  Then, for every $\vec u  \in D$ such that
  \( \dkl(\D(f, S) \| \D(\vec u, S)) \leq r^*_k +
  e^{-\wt{\Omega}\left(\frac{d^3 M}{\alpha^2} + B\right)} \cdot \left(\frac{1}{\eps}\right)^{-\Omega(\log(d/\alpha))} \,, \)
  it holds that $\dtv(\D(f, K), \D(\vec u, K)) \leq \eps$.
\end{theorem}

  The first step in proving 
\Cref{thm:low_dimensional_information_theoretic} is to understand the approximation
error as a function of the degree that we use when we have access to the 
population distribution $\D(f, S)$. This is established in the following lemma 
whose proof can be found in \Cref{sec:app:lem:taylor_approximation}.

\begin{lemma}[Approximation of Log-density]
  \label{lem:taylor_approximation}
  Let $K \subseteq \R^d$ be a convex set centered at the origin $\vec{0}$ of diameter
  $\diam_\infty(K) \leq R$ and let $f \in \LL(B, M)$ be a function
  supported on $K$.
  There exists polynomial
  $p(\bx)   = \vec v^T \vec m_k(\bx)$ such that for every $S \subseteq K$ it holds
  \( \dkl(\D(f,S)\| \D(\vec v, S)) \leq 2 \left(\frac{15 M R d}{k}\right)^{k+1} \)
  and
   \(
  \snorm{\infty}{\1_K p} \leq
  2 B
  + \left(\frac{15 M R d}{k}\right)^{k+1}. \)
\end{lemma}

  From \Cref{lem:taylor_approximation} there exists $\vec v$ such that
  $\snorm{\infty}{\1_K \vec v^T \vec m_k(\bx) } \leq
  2 B + \left(15 M d/k\right)^{k+1}$.
  Moreover, from the same lemma we have that
  $
  \min_{\vec w \in D} \dkl(\D(f, S)\| \D(\vec w, S))
  \leq \dkl(\D(f, S)\| \D(\vec v, S))
  \leq 2 \left(15 M d/k\right)^{k+1}
  $.
  To simplify notation set $r_k = 2 \left(15 M d/k\right)^{k+1}$.
  Now, let $q(\bx) = \vec u^T\vec m_k(\bx)$ be any approximate minimizer in $D$ of
  the KL-divergence between $\D(q, S)$ and $\D(f, S)$ that
  satisfies
  \[ \dkl(\D(f,S) \| \D(\vec u, S)) \leq \min_{\vec w \in D} \dkl(\D(f, S)\| \D(\vec w, S))
  + \bar{\eps}
  \leq r_k + \bar{\eps}.\]
  From Pinsker's inequality and the subadditivity of the square root we get:
$\dtv(\D(f, S), \D(\vec u, S)) \leq \sqrt{r_k} + \sqrt{\bar{\eps}}$. Our next
step is to relate the conditional total variation $\dtv(\D(f, S), \D(\vec u, S))$
with the global total variation $\dtv(\D(f, K), \D(\vec u, K))$. For this we
develop a novel extrapolation technique based on anti-concentration of polynomial
functions. In particular we use the following theorem.

\begin{theorem} [Theorem 2 of \cite{CarberyW01}]
  \label{thm:cw_convex}
    Let $K = [0, 1]^d$ and let $p:\R^d \mapsto \R$ be a polynomial of degree
  at most $k$. If $q \ge 1$, then there exists a constant
  $C$ such that for any $\gamma > 0$ it holds
\(
  \left(\int_K |p(\bx)|^{q/k} \d \bx \right)^{1/q}
  \int_K \1\{|p(\bx)| \leq \gamma \} \d \bx
  \leq C \gamma^{1/k} ~ \min(q, d) \,.
  \)
\end{theorem}

\noindent This result is crucial for extrapolation because it
can be used to bound the behavior of a polynomial function even outside the
region from which we get the samples. This is the main idea of the following
lemma which is one of the main technical contributions of the paper and we
believe that it is of independent interest.

\begin{lemma}[Distortion of Conditioning]
  \label{lem:distance_reduction}
  Let $K = [0, 1]^d$ and let $p, q$ be polynomials of degree at most $k$ such
  that $p, q\in L_\infty(K, B)$. There exists absolute constant $C>0$ such that
  for every $S \subseteq K$ with $\vol(S) > 0$ it holds
  \[
  e^{-2 B} \vol(S)
  ~
  \leq
  \frac{\dtv(\D(p, K), \D(q, K))}{\dtv(\D(p, S), \D(q, S))}
  \leq
  ~
  8 e^{5 B} \frac{(2 C \min(d, 2k))^k}{\vol (S)^{k + 1}}\,.
  \,
  \]
\end{lemma}

\begin{remark}
    Both \Cref{thm:cw_convex} and \Cref{lem:distance_reduction}
  hold for the more general case where $K$ is an arbitrary convex subset of
  $\R^d$. We choose to state this weaker expression for ease of notation.
\end{remark}

  Unfortunately it is still not clear how to apply \Cref{lem:distance_reduction} because 
  it  assumes that both distributions that we are
comparing have as log-density a bounded degree polynomial. 
Nevertheless, we can use a sequence of triangle inequalities together with 
Taylor's Theorem (see \Cref{thm:TaylorTheoremApp}) to combine 
\Cref{lem:distance_reduction} and \Cref{lem:distance_reduction} from 
which we can prove \Cref{thm:low_dimensional_information_theoretic} 
as we explain in detail in
\Cref{sec:proof:low_dimension_information_theoretic}. The proof of the
Distortion of Conditioning Lemma is presented in  \Cref{sec:proof:distance_reduction}.

\subsection{Computing the MLE} \label{sec:sgd}

  In this section we describe an efficient algorithm that solves the Maximum
Likelihood problem that we need in order to apply  \Cref{thm:low_dimensional_information_theoretic}. 
  We solve the following  problem: given sample access to the conditional distribution $\D(f, S)$ and 
  some fixed degree $k$, our algorithm finds a polynomial $p$ of degree $k$ that
approximately minimizes the divergence $\dkl(\D(f, S) \| \D(\vec{u}, S))$.

\begin{theorem} \label{thm:KL_minimization}
    Let $f : [0, 1]^d \to \R$ and $S \subseteq [0, 1]^d$ with
  $\vol(S) \ge \alpha$. Fix a degree $k \in \N$, a parameter $C > 0$, 
  and define
  $D = \{\vec v ~ : ~ \max_{\vec{x} \in [0, 1]^d} |\bv^T \vec m_k(\bx)|  \leq C \}$.
  There exists an algorithm that draws $N =
  2^{O(dk)} (C^2/\eps)^2$
  samples from $\D(f, S)$, runs in time
  $2^{O(d k + C)}/(\alpha\eps^2)$, and outputs
  $\hat{\vec{v}} \in D$ such that
  $\dkl(\D(f, S) \| \D(\hat{\vec{v}}, S)) \leq \min_{\bu \in D}\dkl(\D(f, S) \| \D(\vec u, S)) + \eps$, 
 with probability at least $ 99 \%$.
\end{theorem}

  The algorithm that we use for proving \Cref{thm:KL_minimization} is
Projected Stochastic Gradient Descent with projection set $D$. In order to prove
the guarantees of \Cref{thm:KL_minimization} we have to prove: (1) an
upper bound on the number of steps that the PSGD algorithm needs, (2) find an
efficient procedure to project to the set $D$. For the latter we can use the
celebrated algorithm by Renegar for the existential theory of reals
\cite{Ren92b,Ren92a}, as we explain in detail in \Cref{sec:proof:KL_minimization}. 
To analyze PSGD we use the following lemma.

\begin{lemma}[Theorem 14.8 of \cite{ShalevB14}]
  \label{lem:sgd_convex}
  Let $R, \rho> 0$. Let $f$ be a convex function, $D\subseteq \R^d$ be a
  convex set of bounded diameter, $\diam_2(D) \leq R$, and let $\vec w^* \in
  \argmin_{\bw \in D} f(\bw)$.  Consider the following Projected Gradient
  Descent (PSGD) update rule
  $\bw_{t+1} = \mathrm{proj}_D(\bw_{t} - \eta \vec v_t)$,
  where $\vec v_t$ is an unbiased estimate of $\nabla f(\bw)$.  Assume that
  PSGD is run for $T$ iterations with $\eta = \sqrt{R^2\rho^2/T}$.  Assume
  also that for all $t$, $\snorm{2}{\vec v_t} \leq \rho$ with probability
  $1$.  Then, for any $\eps > 0$, in order to achieve $\E[f(\bar{\bw})] -
  f(\bw^*) \leq \eps $ it suffices that $T\geq R^2 \rho^2/\eps^2$.
\end{lemma}

From the above lemma we can see that it remains to find an upper bound on the
diameter of the set $D$ and an upper bound on the norm of the stochastic
gradient $\norm{\vec{v}_t}$. The latter follows from some algebraic calculations
whereas the first one from tight bounds on the coefficients of a polynomial with
bounded values \cite{SBGK18}. For the full proof of \Cref{thm:KL_minimization}, see 
\Cref{sec:proof:KL_minimization}.
 
\subsection{Putting Everything Together -- The Proof of \Cref{thm:low_dimensional_algorithm}}
  From \Cref{thm:low_dimensional_information_theoretic} we have that
  if we fix the degree $k = O(d^3M/\alpha^2 + B) + 2 \log(1/\eps)$ then it
  suffices to optimize the function $L(\bv)$ of 
  \Cref{eq:kullback_leibler_objective} constrained in the convex set
  $
  D =
  \left\{
    \vec v \in \R^m~:~ \snorm{\infty}{\1_K \vec v^T \vec m_k(\bx) }
    \leq 3 B
 \right\}.
 $
 From \Cref{thm:low_dimensional_information_theoretic} we have that a
 vector $\vec v$ with optimality gap $2^{-\wt{\Omega}(d^3M/\alpha^2 +B)}
 (1/\eps)^{-\Omega(\log(d/\alpha)}$ achieves the extrapolation guarantee
 $\dtv(\D(f, K), \D(\bv_T, K)) \leq \eps$.  From
 \Cref{thm:KL_minimization} we have that there exists an algorithm that
 achieves this optimality gap with
 sample complexity
 \(
 N = B^2 2^{O(dk)} 2^{\wt{O}(d^3 M/\alpha^2 + B)} (1/\eps)^{O(\log(d/\alpha))}
 = 2^{\wt{O}(d^4 M /\alpha^2 + B d)} (1/\eps)^{d + \log(1/\alpha)}\ .
 \)
 
\bibliographystyle{alpha}
\bibliography{refs}

\appendix

\section{Multidimensional Taylor's Theorem} \label{sec:app:taylor}

  In this section we present the Taylor's theorem for multiple dimensions and
we prove \Cref{thm:TaylorTheoremApp}. We remind the following notation 
from the preliminaries section 
$\vec{x}^{\vec{\alpha}} = x_1^{\alpha_1} \cdot x_2^{\alpha_2} \cdots x_d^{\alpha_d}$.

\begin{theorem}[Multi-Dimensional Taylor's Theorem] \label{thm:TaylorTheorem}
    Let $S \subseteq \R^d$ and $f : S \to \R$ be a $(k + 1)$-times
  differentiable function, then for any $\vec{x}, \vec{y} \in S$ it holds
  that
  \begin{align*}
    f(\vec{y}) = \sum_{\vec{\alpha} \in \N^d, \abs{\vec{\alpha}} \le k} \frac{\dr_{\vec{\alpha}} f(\vec{x})}{\vec{\alpha}!} (\vec{y} - \vec{x})^{\vec{\alpha}} + H_k(\vec{y}; \vec{x}), ~~~~ \text{with}
  \end{align*}
  \begin{align*}
    H_k(\vec{y}; \vec{x}) = \sum_{\vec{\beta} \in \N^d, \abs{\vec{\beta}} = k + 1} R_{\vec{\beta}}(\vec{y}; \vec{x}) (\vec{y} - \vec{x})^{\vec{\beta}} ~~ \text{and} ~~ R_{\vec{\beta}}(\vec{y}; \vec{x}) = \frac{\abs{\vec{\beta}}}{\vec{\beta}!} \int_0^1 (1 - t)^{\abs{\vec{\beta}} - 1} \dr_{\vec{\beta}} f(\vec{x} + t (\vec{y} - \vec{x})) dt.
  \end{align*}
\end{theorem}

  We now provide a proof of \Cref{thm:TaylorTheoremApp}.

\begin{proof}[Proof of \Cref{thm:TaylorTheoremApp}]
    We start by observing that
  \begin{align*}
    \abs{f(\vec{y}) - \bar{f}_k(\vec{y}; \vec{x})} \le \left(\sum_{\vec{\beta} \in \N^d, \abs{\vec{\beta}} = k + 1} \frac{1}{\vec{\beta}!} \right) \cdot R^{k + 1} \cdot W.
  \end{align*}
  This inequality follows from multidimensional Taylor's Theorem by
  some simple calculations. Now to show wanted result it suffices to show that
  $\sum_{\vec{\beta} \in \N^d, \abs{\vec{\beta}} = k + 1} \frac{1}{\vec{\beta}!} \le \left(\frac{15 d}{k} \right)^{k + 1}$. To prove
  the latter we first show that $\min_{\vec{\beta} \in \N^d, \abs{\vec{\beta}} = k + 1} \vec{\beta}! = (\ell!)^{d - r} ((\ell + 1)!)^{r}$ where
  $\ell = \lfloor \frac{k + 1}{d} \rfloor$ and $r = k + 1 \pmod{d}$. We prove
  this via contradiction, if this is not true then the minimum
  $\min_{\vec{\beta} \in \N^d, \abs{\vec{\beta}} = k + 1} \vec{\beta}!$ is
  achieved in multi-index $\vec{\beta}$ such that there exist $i, j \in [d]$
  such that $\beta_i < \ell$ and $\beta_j > \ell + 1$. In this case we define
  $\vec{\beta}'$ to be equal to $\vec{\beta}$ except for
  $\beta'_i = \beta_i + 1$ and $\beta'_j = \beta_j - 1$. In this case we get
  $\vec{\beta}'! < \frac{\beta_j}{\beta_i + 1} \vec{\beta}'! = \vec{\beta}!$,
  which contradicts the optimality of $\vec{\beta}$. Therefore we have that
  $\sum_{\vec{\beta} \in \N^d, \abs{\vec{\beta}} = k + 1} \frac{1}{\vec{\beta}!} \le \binom{d + k + 1}{k + 1} \frac{1}{((\ell + 1)!)^{d}}$. Now via upper bounds from Stirling's approximation we get that
  $\binom{d + k + 1}{k + 1} \frac{1}{((\ell + 1)!)^{d}} \le \frac{e^{k+1} \left(1+\frac{d}{k+1} \right)^{k+1}}{((k+1)/d)^{k+1} e^{-k-1}} \le \left(\frac{e^2 d}{k}\right)^{k+1} \left(1 + \frac{k}{k+1}\right)^{k+1}$
  and the Theorem follows from simple calculations on the last expression.
\end{proof}

\section{Missing Proofs for Single Dimensional Densities} \label{sec:app:singleDimensional}

  In this section we provide the proof of the theorems presented in  \Cref{sec:one-dimensional}.

\subsection{Proof of \Cref{thm:one_dimensional_exact_extrapolation}}
\label{sec:proof:one_dimensional_exact_extrapolation}

We are going to use the following result that bounds the
error of Hermite polynomial interpolation, wherein besides matching
the values of the target function the approximating polynomial  also matches  its derivatives.  The following
theorem can be seen as a generalization of Lagrange interpolation, where the
interpolation nodes are distinct, and Taylor's remainder theorem where we
find a polynomial that matches the first $k$ derivatives at a single node.
\begin{lemma}[Hermite Interpolation Error]
  \label{lem:hermite-interpolation-error}
  Let $x_1, \ldots, x_s$ be distinct nodes in $[a,b]$
  and let $m_1,\ldots, m_s \in \N$ such that
  $\sum_{i=1}^s m_i = k+1$.
  Moreover, let $f$ be a $(k+1)$ times continuously differentiable function over
  $[a,b]$ and $p$ be a polynomial of degree at most $k$ such that for each
  $x_i$
  $$
  p(x_i) = f(x_i)~~~ p'(x_i) = f'(x_i)~~~ \ldots
  ~~~
  p^{(m_i-1)} (x_i) = f^{(m_i-1)}(x_i)\,.
  $$
  Then for all $x\in[a,b]$, there exists $\xi \in (a,b)$ such
  that
  $$
  f(x) - p(x) = \frac{f^{(k+1)}(\xi)}{(k+1)!}
  \prod_{i=1}^s (x - x_i)^{m_i}\,.
  $$
\end{lemma}
We are also going to use the following upper bound on Kullback-Leibler
divergence.  For a proof see Lemma 1 of \cite{BS91}.
\begin{lemma}
  \label{lem:kl_quadratic_upper}
  Let $\mcal{P}, \mcal{Q}$ be distributions on $\R$ with corresponding density
  functions $p, q$.  Then for any $c>0$ it holds
  $$
  \dkl(\mcal{P}\| \mcal{Q})
  \leq e^{\snorm{\infty}{\log(p(x)/q(x)) -c}}
  \int p(x) \lp(\log \frac{p(x)}{q(x)} - c \rp)^2 \d x\, .
  $$
\end{lemma}

Before, the proof of \Cref{thm:one_dimensional_exact_extrapolation} we
are going to show a useful lemma.  Let $f, g$ be two density functions such
that $\dkl(f\|g) > 0$ and let $r$ another function $r$ that lies strictly
between the two densities $f, g$.  The following lemma states that  after we
normalize $r$ to become a density function $\bar{r}$ we get that $\bar{r}$
is closer to $f$ in Kullback-Leibler divergence than $g$.
\begin{lemma}[Kullback-Leibler Monotonicity]
  \label{lem:kl-monotonicity}
    Let $f, g$ be density functions over $\R$ such that the measure defined by
  $f$ is absolutely continuous with respect to that defined by $g$, i.e. the
  support of $f$ is a subset of that of $g$. Let also $r$ be an integrable
  function such that $r(x) \geq 0$, for all $x \in \R$, and moreover, for all
  $x \in \R \setminus Z$
  \[ f(x) \leq r(x) < g(x) \quad \text{ or } \quad g(x) < r(x) \leq f(x) \]
  where $Z$ is a set that has measure $0$ under both $f$ and $g$. Then, if
  $ \bar{r}(x) = r(x)/\int r(x) \d x$ is the density function corresponding to
  $r(\cdot)$, it holds that
  \[ \dkl(f \| \bar{r}) < \dkl(f \| g). \]
\end{lemma}
\begin{proof}
  To simplify notation we are going to assume that the support of $f$ is the entire
  $\R$, and we define the sets $A_< = \{x\in\R: r(x) < g(x)\}$, $A_> = \{x \in \R
  : r(x) > g(x) \}$.  In the following proof we are going to ignore the
  measure zero set where the assumptions about $g, r, f$ do not hold.  Denote
  $C = \int r(x) \d x = 1 - \int (g(x) - r(x)) \d x $.
  We have
  \begin{align*}
    \dkl(f\|g) - \dkl(f\|\bar{r})
  &= \int f(x) \log \frac{ \bar{r}(x)}{g(x)}
  \ \d x
  = \int f(x) \log \frac{r(x)}{g(x)}\ \d x
  - \log C
  \\
  &\geq \int r(x) \log \frac{r(x)}{g(x)} \ \d x
  - \log C\, ,
  \end{align*}
  where for the last inequality we used the fact that $f(x) < r(x)$ for all
  $x \in A_{<}$ and $f(x) > r(x)$ for all $x \in A_{>}$.  Using the
  inequality $\log (1 + z) \leq z $ we obtain
  $-\log C \geq  \int (g(x) - r(x) ) \d x. $ Using this fact we obtain
  \begin{align*}
    \dkl(f\|p) - \dkl(f\|\bar{r})
  &\geq \int r(x) \log \frac{r(x)}{g(x)} \ \d x + \int (g(x) - r(x))\ \d x
  \\
  &\geq \int \left(r(x) \log \frac{r(x)}{g(x)} + g(x) - r(x) \right) \d x.
  \end{align*}
  To finish the proof we observe that $r(x) \log(r(x)/g(x)) + g(x) - r(x) > 0$
  for every $x $.  To see this we rewrite the inequality as
  $
  \log \frac{r(x)}{g(x)} -1 + \frac{g(x)}{r(x)} > 0.
  $
  To prove this we use the inequality $\log(z) - 1 +1/z > 0$
  for all $z \neq 1$.
\end{proof}

We are now ready to prove the main result of this section: \Cref{thm:one_dimensional_exact_extrapolation}.

\subsubsection*{Proof of \Cref{thm:one_dimensional_exact_extrapolation}}
Recall that
  $$
  p = \argmin_{q \in \mathcal{Q}_k} \dkl(\D(f,S) \| \D(q, S))\,.
  $$

  To simplify notation, we define the functions $\phi_f(x) = f(x) - \psi(f,S)$
and $\phi_p(x) = p(x) - \psi(p,S)$.  Notice that these are the log
  densities of the conditional distributions $\D(f,S)$ and $\D(p, S)$ on the set
  $S$, viewed as functions over the entire interval $I$ (i.e.~they
  are the conditional log densities without the indicator $\1_S(x)$).
  Notice that $\phi_p$ is a polynomial of degree at most $k$.
  Let $g(x) = \phi_f(x) - \phi_p(x)$.
  We first show the following claim.
  \begin{claim}
    \label{clm:information-projection-roots}
      The equation $g(x) = 0$ has at least $k+1$ roots in $S$, counting
    multiplicities.
  \end{claim}
  \begin{proof}[Proof of \Cref{clm:information-projection-roots}]
  To reach a contradiction, assume that it has $k$ (or fewer) roots.
  Let $\xi_1, \ldots, \xi_s \in I $ be the
  distinct roots of $g(x) = 0$ ordered in increasing order and let $m_1,\ldots,
  m_s$ be their multiplicities.
Denote by $I_0, \ldots, I_s$ the partition of $I$
  using the roots of $g(x)$, that is
  $$
  I_0 =
  (-\infty, \xi_1] \cap I
  = [\xi_0, \xi_1]
  ,\
  I_1 = [\xi_1, \xi_2] ,\
  \ldots,\
  I_s = [\xi_s, \xi_{s+1}]
  =
  [\xi_s, +\infty) \cap I
  \,.
  $$
  Let $q$ be the polynomial that has the same roots as $g(x)$ and also the same
  sign as $g(x)$ in every set $I_j$ of the partition. We claim that there exists
  $\lambda_j > 0$ such that, for every $j$ and $x \in \mrm{int}(I_j)$, the
  expression $\phi_f(x) - (\phi_p(x) + \lambda q(x))$ has the same sign as
  $\phi_f(x) - \phi_p(x)$ and also 
  $|\phi_f(x) - (\phi_p(x) + \lambda_j q(x))| < |\phi_f(x) - \phi_p(x)|$. 
  Indeed, fix an interval $I_j$ and without loss of generality assume that 
  $g(x) > 0$ for all $x \in I_j \setminus \{\xi_{j}, \xi_{j+1}\} $. Then it
  suffices to show that there exists $\lambda_j > 0$ such that 
  $0 < g(x) - \lambda_j q(x) < g(x)$.  Since $q(x) > 0$ for every 
  $x \in \mrm{int}(I_j)$, we  need to choose $\lambda_j < g(x)/q(x)$. Since
  $\xi_{j}$ is a root of the same multiplicity of both $g(x)$ and $q(x)$ and
  $g(x), q(x) > 0$ for all $x \in \mrm{int}(I_j)$ we have 
  $\lim_{x \to \xi_j^+} \frac{g(x)}{q(x)} = a > 0$. Similarly, we have 
  $\lim_{x \to \xi_{j+1}^-} \frac{g(x)}{q(x)} = b > 0$. We can now define the
  following function
  \begin{align*}
    h(x) =
    \begin{cases}
      a, &x = \xi_j \\
      g(x)/q(x), &\xi_j < x < \xi_{j+1} \\
      b, &x = \xi_{j+1}
    \end{cases}\,.
  \end{align*}
  We showed that $h(x)$ is continuous in $I_j = [\xi_j, \xi_{j+1}]$ and therefore
  has a minimum value $r_j> 0$ in the closed interval $I_j$.  We set $\lambda_j =
  r_j$.  With the same argument as above we obtain that for each interval $I_j$ we
  can pick $\lambda_j > 0$.  Since the number of intervals in our partition is
  finite we may set $\lambda = \min_{j=0,\ldots,s} \lambda_j$ and still have
  $\lambda > 0$.

  We have shown that the polynomial $r(x) = \phi_p(x) + \lambda q(x)$ is
  almost everywhere, that is apart from a measure zero set, \emph{strictly}
  closer to $\phi_f(x)$.  In particular, we have that $\phi_p(x)$ for every
  $x \in [0,1]$, $|\phi_f(x) - r(x)| \leq |\phi_f(x) - \phi_p(x)|$ and for
  every $x \in S \setminus \{\xi_0, \ldots, \xi_{s+1}\}$ it holds $|\phi_f(x)
  - r(x)| < |\phi_f(x) - \phi_p(x)|$.  Moreover, by construction we have that
  $\phi_f(x) - r(x)$ and $\phi_f(x) - p(x)$ are always of the same sign.
  Finally, the degree of $r(x)$ is at most $k$.  Using
  \Cref{lem:kl-monotonicity} we obtain that $\dkl(\D(f, S) \| \D(r, S) )
  < \dkl(\D(f,S) \| \D(p, S) )$, which is impossible since we know that $p$
  is the polynomial that minimizes the Kullback-Leibler divergence.
\end{proof}
  We are now ready to finish the proof of our lemma.  Using
  \Cref{lem:hermite-interpolation-error} and
  \Cref{clm:information-projection-roots} we have that
$\phi_p$ and $\phi_f$ are close not only in $S$ but in the whole interval $I$.
  In particular, for every $x \in I$ it holds
  \begin{equation}
    \label{eq:kl_minimizer_pointwise}
    \left|f(x) - \log \int_S e^{f(x)} \d x - p(x)
  + \log \int_S e^{p(x)} \d x \right| \leq \frac{M}{(k+1)!} R^{k+1} := W_k \,.
\end{equation}
  Using the above bound together with \Cref{lem:kl_quadratic_upper},
  where we set $c = \log \frac{\psi(f,I) \psi(p,S)}{\psi(p,I) \psi(f,S)}$,
we obtain
  \[ \dkl(\D(f,I)\| \D(p,I)) \leq e^{W_k} W_k^2. \]

\subsection{Proof of \Cref{thm:one_dimensional_aproximate_extrapolation}}
\label{sec:proof:one_dimensional_aproximate_extrapolation}

  We first show that the minimizer of the Kullback-Leibler divergence
belongs to the set $D_k$. Let $q^* = \argmin_{q \in \mathcal{Q}_k} \dkl(\D(f,S) \|
\D(q, S))$ be the minimizer over the set of degree $k$ polynomials.
From the assumption that $f \in L_\infty(I,B)$ we have that
$e^{-B} \alpha \leq \int_S e^{f(x)} \d x \leq e^B$ and therefore
$|\psi(f, S)| \leq B + \log(1/\alpha)$. We know from 
\Cref{eq:kl_minimizer_pointwise} that for all $x \in I$ it holds
\[ |q(x) - \psi(q, S) - f(x) + \psi(f, S)| \leq W_k = \frac{M^{k+1}}{(k+1)!}. \]
In particular, for $x = 0$ using the above inequality we have
$|\psi(q^*, S)| \leq W_k + |\psi(f,S)| + |f(0)| \leq W_k + 2 B + \log(1/\alpha) $.
Therefore, $q^* \in D_k$. From the Pythagorean identity of the information
projection we have that for any other $q \in D_k$ it holds
\[ \dkl(\D(f, S) \| \D(q, S))
  = \dkl(\D(f, S) \| \D(q^*, S)) + \dkl(\D(q^*, S) \| \D(q, S)). \]
Therefore, from the definition of $q$ as an approximate minimizer with
optimality gap $\eps$ we have that $\dkl(\D(q^*, S) \| \D(q, S))  \leq \eps$ and
from Pinsker's inequality we obtain
$\dtv(\D(q^*, S), \D(q, S)) \leq \sqrt{\eps}$.  Using the triangle inequality,
we obtain
\[ \dtv(\D(q,I), \D(f, I))
  \leq
  \dtv(\D(q,I), \D(q^*,I))
  +
  \dtv(\D(q^*,I), \D(f,I)) \]
From \Cref{thm:one_dimensional_exact_extrapolation} and Pinsker's
inequality we obtain that
$\dtv(\D(q^*,I), \D(f,I)) \leq e^{W_k/2} W_k$.
Moreover, from \Cref{lem:distance_reduction} we have that
$\dtv(\D(q, I), \D(q^*, I)) \leq 4 e^{10 B} (2 C/\alpha)^{k+8} \sqrt{\eps} $,
where $C$ is the absolute constant of \Cref{thm:cw_convex}.

\section{Missing Proofs for Multi-Dimensional Densities} \label{sec:app:multiDimensional}

  In this section we provide the proofs of the lemmas and theorems presented in
Section \ref{sec:multivariate}.

\subsection{Proof of \Cref{lem:taylor_approximation}} \label{sec:app:lem:taylor_approximation}
We remind that $\mathcal{Q}_{d, k}$ is the space of polynomials of degree at most
$k$ with $d$ variables and zero constant term, where we might drop $d$ from the
notation if it is clear from context.

  The bound on the norm $\snorm{\infty}{\1_K p}$ follows directly from Taylor's
  theorem. In particular, using \Cref{thm:TaylorTheoremApp} we obtain that
  there exists the Taylor polynomial $f_k(\cdot;\vec 0)$ of degree $k$ around
  $\vec 0$ satisfies $\snorm{\infty}{\1_K (f - f_k)} \leq (15 M R d/k)^{k + 1}$.

The bound on the Kullback-Leibler now follows directly from the following
simple inequality that bounds the Kullback-Leibler divergence in terms of the
$\ell_\infty$ norm of the log-densities.
  We have
  \begin{align}
    \label{eq:kl_infty_upper_bound}
  \dkl(\D(f, S) \| \D(g, S))
  &= \int_S \D(f, S; \bx) (f(\bx) - g(\bx)) \d \bx
  + \log\left(\int_S e^{g(\bx)} \d \bx \right) -
  \log\left(\int_S e^{f(\bx)} \d \bx \right)
  \nonumber
  \\
  &\leq  \snorm{\infty}{\1_S (f - g)}
  + \log\left(\int_S e^{f(\bx) + \snorm{\infty}{\1_S(f-g)}} \d \bx \right) -  \log\left(\int_S e^{f(\bx)} \d \bx \right)
  \nonumber
  \\
  &\leq  2 \snorm{\infty}{\1_S (f - g)}\,.
\end{align}

The polynomial provided by \Cref{thm:TaylorTheoremApp} does not
necessarily have zero constant term.  We can simply subtract this constant and
show that the $L_\infty$ norm of the resulting polynomial does not grow by a
lot.  Using the triangle inequality we get
  \begin{align*}
  \snorm{\infty}{\1_K (f_k - f_k(\vec{0}))}
  &\leq
  \snorm{\infty}{\1_K f}
  +
  \snorm{\infty}{\1_K (f_k - f)}
    + |f_k(\vec 0)|
    \\
  &\leq
  2 B
  +
  M \left(\frac{15 R d}{k}\right)^{k+1}
\,.
  \end{align*}
  Finally, we observe that the polynomials $f_k$ and $f_k-f_k(\vec 0)$
  correspond to the same distribution after the normalization, therefore
  it still holds $\dkl(\D(f, S)\| \D(f_k - f_k(\vec 0), S)) =
  \dkl(\D(f, S) \| \D(f_k, S)) \leq 2 (15 M R d/k)^{k+1}$. \qed  
  
\subsection{Proof of Distortion of Conditioning: \Cref{lem:distance_reduction}} \label{sec:proof:distance_reduction}

The Distortion of Conditioning Lemma contains two inequalities; an upper bound 
and a lower bound on $\dtv(\D(p,K), \D(q,K))/\dtv(\D(p, K), \D(q, K))$. We begin
our proof with the upper bound and then we move to the lower bound.
\smallskip

\noindent \textbf{Upper Bound.} We first observe that for every set 
$R \subseteq K$ it holds
\begin{equation}
  \label{eq:volume_infty}
  e^{-\snorm{\infty}{\1_K p}} \vol(R)
  \leq
  \int_R e^{p(\bx)} \d \bx
  \leq e^{\snorm{\infty}{\1_K p}} \vol(R)
\end{equation}
which implies that $|\psi(p, R)|
\leq \snorm{\infty}{\1_K p} + \log(1/\vol(R))$. 
\smallskip

\noindent Now to prove the upper bound on the ratio 
$\dtv(\D(p,K), \D(q,K))/\dtv(\D(p, K), \D(q, K))$, we will prove an upper bound
on $\dtv(\D(p,K), \D(q,K))$ and a lower bound on $\dtv(\D(p, K), \D(q, K))$. We 
begin with the lower bound on $\dtv(\D(p, K), \D(q, K))$.
\begin{align}
  \label{eq:tvd_volume_lower_bound}
2 \dtv(\D(p,S), \D(q,S))
&= \int_S
\abs{
\frac{e^{p(\bx)}}{e^{\psi(p,S)}} -
\frac{e^{q(\bx)}}{e^{\psi(q,S)}}
}
  \d \bx
  \nonumber
  \\
&\geq
\min_{\bx \in S}\left(\frac{e^{p(\bx)}}{e^{\psi(p,S)}}\right)
  \int_S
  \abs{
  1 - \frac{e^{-p(\bx)}}{e^{-\psi(p,S)}}
    \cdot
    \frac{e^{q(\bx)}}{e^{\psi(q, S)}}
  }
  \d \bx
  \nonumber
  \\
  &\geq
  \frac{e^{-2 \snorm{\infty}{\1_K p}}}{\vol(S)}
  \int_S \left| 1 - e^{r(\bx)} \right| \d \bx\, ,
\end{align}
where $r(\vec x) = q(\vec x) - \psi(q, S) - (p(\vec x) - \psi(p, S))$. For
some $\gamma > 0$ we define the set
$Q = K \cap \{\vec  z : |r(\vec z)|\leq \gamma \}$. Using
\Cref{thm:cw_convex} for the degree $k$ polynomial $r(\vec x)$ and
setting $q = 2k$,
$\gamma = \left(\frac{\vol(S)}{2 C \min\{d, 2k\}}\right)^k \sqrt{\int_K (r(x))^2 \d \bx}$,
we get that $\vol \left(Q\right) \leq \vol(S)/2$. Using these definitions we have
\begin{align*}
\int_S |1- e^{r(\bx)}| \d \bx
\geq
\int_{S \setminus Q}
|1- e^{r(\bx)}|
\d \bx
\geq
\frac{\vol(S)}{2}
\min_{\bx \in S\setminus Q}|1-e^{r(\bx)}|
\,.
\end{align*}
Since $|r(\bx)| \geq \gamma $ for all $x \in S \setminus Q$ we have
that if $\gamma \geq 1$ then from the inequality $|1-\me^x | \geq 1/2$ for 
$|x| > 1$ and from \Cref{eq:tvd_volume_lower_bound} we obtain
\[ 2 \dtv(\D(p,S), \D(q, S)) \geq \frac{ e^{-2 \snorm{\infty}{\1_K p}} }{4}
\]
If $\gamma < 1$ then we can use the inequality
$|1-\me^{x}|\geq |x|/2$  for $|x| \leq 1$ together with
\Cref{eq:tvd_volume_lower_bound} to get
\begin{align*}
  2 \dtv(\D(p,S), \D(q, S))
  & \geq \frac{e^{-2 \snorm{\infty}{\1_K p}}}{4} \vol(S) \cdot \gamma. \end{align*}
and hence for every value of $\gamma$ we have that
\begin{align} \label{eq:proof:distortion:lowerBoundOnConditionalTV}
  2 \dtv(\D(p,S), \D(q, S))
  & \geq \frac{e^{-2 \snorm{\infty}{\1_K p}}}{4} \min\{\vol(S) \cdot \gamma, 1\}. \end{align}
Next we find an upper bound on $\dtv(\D(p,K), \D(q,K))$. In particular, we are
going to relate the total variation distance of $\D(p)$ and $\D(q)$ with the
integral $\int_K (r(\bx))^2 \d \bx$. Applying \Cref{lem:kl_quadratic_upper}
with $c = -(\psi(q,K) - \psi(p, K)) + (\psi(q, S) - \psi(p,S))$ we have that
\begin{align*}
  \dkl(\D(p, K)\| \D(q, K))
  &\leq
  e^{\snorm{\infty}{\1_K r}}
  \int_K \D(p, K; \bx) (r(\bx))^2 \d \bx
  \leq
  e^{\snorm{\infty}{\1_K r} + 2 \snorm{\infty}{1_K p}}
  \int_K (r(\bx))^2\, \d \bx .
\end{align*}
From \Cref{eq:volume_infty} we obtain that
$\snorm{\infty}{\1_K r} \leq 2 \snorm{\infty}{\1_K p} + 2 \snorm{\infty}{\1_K q} + 2 \log(1/\vol(S))$.
Now, using Pinsker's and the above inequality we obtain
\begin{align*}
  \dtv(\D(p, K), \D(q, K))
  &\leq
  \sqrt{\dkl(\D(p, K)\| \D(q, K))}
  \leq
  \frac{e^{2 \snorm{\infty}{\1_K p} + \snorm{\infty}{\1_K q}}}{\vol(S)}
  \sqrt{\int_K (r(\bx))^2 \d \bx}
  \\
  &\leq
  e^{2 \snorm{\infty}{\1_K p} + \snorm{\infty}{\1_K q}}
  \frac{(2 C \min\{d, 2k\})^k}{\vol(S)^{k+1}} \ \gamma.
\end{align*}
which implies our desired upper bound on the ratio
$\frac{\dtv(\D(p, K), \D(q, K))}{\dtv(\D(p, S), \D(q, S))}$, using 
\Cref{eq:proof:distortion:lowerBoundOnConditionalTV} and $\vol(S) \le 1$.
\medskip

\noindent \textbf{Lower Bound.} We now show the lower bound on 
$\dtv(\D(p, K), \D(q, K))/\dtv(\D(p,S), \D(q,S))$.
We have that
\begin{align*}
  2 \dtv(\D(p,S), \D(q,S))
  &=
  \int 1_S(\bx) \left|
  \frac{\D(p, K;\bx)}{\D(p, K; S)} -
  \frac{\D(q, K; \bx)}{\D(q, K; S)} \right| \d \bx
  \\
  &= \int 1_S(\bx) \left|
  \frac{\D(p, K; \bx)}{\D(p, K; S)}
  - \frac{\D(q, K; \bx)}{\D(p, K; S)}
  +  \frac{ \D(q, K; \bx) }{\D(p, K;S)}
  - \frac{\D(q, K; \bx)}{\D(q, K;S)}
  \right| \d \bx
  \\
  &\leq
  \frac{1}{\D(p, K; S)}
   \int 1_S(\bx) \left| \D(p, K;\bx) - \D(q, K;\bx) \right|
  \d \bx
  \\
  &+  \int \1_S(\bx)
  \left| \frac{\D(q, K;\bx)}{\D(p, K;S)} -
  \frac{\D(q, K;\bx)}{\D(q, K; S)}
  \right| \d \bx
  \\
  &\leq
  \frac{1}{\D(p, K; S)} \dtv(\D(p, K),\D(q, K))
  +  \left| \frac{\D(q, K;S)}{\D(p, K;S)}
  - 1
  \right|
  \\
  &\leq
  \frac{2}{\D(p, K;S)} \dtv(\D(p, K),\D(q, K))\,,
\end{align*}
where for the last step we used the
fact that $|\D(p,K;S) - \D(q, K;S)| \leq \dtv(\D(p,K), \D(q,K))$.
Using again \Cref{eq:volume_infty} we obtain that
\begin{align*}
 \D(p, K; S) & = e^{\psi(p, S)-\psi(p,K)} = \left( \int_S e^{p(\bx)} \d \bx \right) \cdot \left( \int_K e^{p(\bx)} \d \bx \right)^{-1} \\
             & \geq \frac{e^{-\snorm{\infty}{\1_K p}} \vol(S)}{e^{\snorm{\infty}{\1_K p}}\vol(K)} \geq e^{-2 B} \frac{\vol(S)}{\vol(K)},
\end{align*}
and since $\vol(K) = 1$ the wanted bound of the lemma follows. \qed

\subsection{Proof of \Cref{thm:low_dimensional_information_theoretic}}
\label{sec:proof:low_dimension_information_theoretic}

From \Cref{lem:taylor_approximation} we obtain that by choosing
  $\vec 0 \in K$, there exists $\vec v$ such that
  $\snorm{\infty}{\1_K \vec v^T \vec m_k(\bx) } \leq
  2 B + \left(15 M d/k\right)^{k+1}$.
  Moreover, from the same lemma we have that
  $
  \min_{\vec w \in D} \dkl(\D(f, S)\| \D(\vec w, S))
  \leq \dkl(\D(f, S)\| \D(\vec v, S))
  \leq 2 \left(15 M d/k\right)^{k+1}
  $.
  To simplify notation set $r_k = 2 \left(15 M d/k\right)^{k+1}$.
  Now, let $q(\bx) = \vec u^T\vec m_k(\bx)$ be any approximate minimizer in $D$ of
  the KL-divergence between $\D(q, S)$ and $\D(f, S)$ that
  satisfies
  $$\dkl(\D(f,S) \| \D(\vec u, S))
  \leq \min_{\vec w \in D} \dkl(\D(f, S)\| \D(\vec w, S))
  + \eps
  \leq r_k + \eps.
    $$
Using the triangle inequality, we have
$$
\dtv(\D(f, K), \D(\vec u, K))
\leq \dtv(\D(f, K), \D(\vec v, K))
+ \dtv(\D(\vec v, K), \D(\vec u, K)).
$$
Using \Cref{lem:distance_reduction} we obtain that
$\dtv(\D(\vec v, K), \D(\vec u, K)) \leq
U \dtv(\D(\vec v, S), \D(\vec u, S))$
where $U = 4 e^{15 B} (2 C d)^{k}/\alpha^{k+3}$.
Using again the triangle inequality we obtain that
\begin{align*}
\dtv(\D(\vec v, K), \D(\vec u, K))
&\leq
U
\left(
\dtv(\D(\vec v, S), \D(f,S)) + \dtv(\D(f,S), \D(\vec u, S))
\right)
\end{align*}
Overall, we have  proved the following important inequality that shows that
we can extend the conditional information to whole set $K$ without increasing
the error by a lot.  In other words, the polynomial with parameters $\vec u$
that we found by (approximately) minimizing the Kullback-Leibler divergence
to the \emph{conditional distribution} is a good approximation on the whole
convex set $K$.
\begin{align*}
\dtv(\D(f, K), \D(\vec u, K))
&\leq
\dtv(\D(\vec v, K), \D(f,K))
+
U
\left(
\dtv(\D(\vec v, S), \D(f,S)) +   \dtv(\D(f,S), \D(\vec u, S))
\right)
\\
&\leq
(U+1) (2 \sqrt{r_k} + \sqrt{\bar{\eps}})
\,.
\end{align*}
where we set $\bar{\eps} =
2^{- \wt{\Omega}\Big(\frac{d^3 M}{\alpha^2} + B \Big)}
(1/\eps)^{-\Omega(\log(d/\alpha))}
$ is the optimality gap of the vector $\vec u$.
For the last inequality we use Pinsker's inequality to
get that $\dtv(\D(\vec v, S), \D(f,S)) \leq
\sqrt{ \dkl(\D(\vec v, S)\| \D(f,S)) }
\leq \sqrt{r_k}$ since the guarantee of \Cref{lem:taylor_approximation}
holds for every subset $R \subseteq K$.
Using again Pinsker's inequality to upper bound the other two total variation distances
we obtain the final inequality.
Substituting the values of $U, L$ we obtain that
\begin{align*}
& (U+1)(2 \sqrt{r_k})
=  O\left(e^{15 B} \frac{(2 C d)^k}{\alpha^{k+3}}
  \frac{(15 M d)^{k/2 + 1/2}}{k^{k/2}}
  \right)
  \\
&=O\left(
  \exp\left( 15 B + \log\left(\frac{\sqrt{M d}}{\alpha^3} \right)
    + k \log \left(C' \frac{\sqrt{d^3 M}}{\alpha}\right)
      - k \log \sqrt{k}
      \right)
  \right),
\end{align*}
where $C'$ is an absolute constant.
Therefore, for $k = O( d^3 M /\alpha^2 + B) + 2 \log(1/\eps)$ it holds that
$U(2 \sqrt{r_k}) \leq \eps/2$.  Moreover, we observe that
$U \sqrt{\bar{\eps}} \leq \eps/2$ and therefore for this value of
$k$ we have
$
\dtv(\D(f, K), \D(\vec u, K)) \leq  \eps/2 + \eps/2 \leq \eps.
$
\qed

\subsection{Proof of \Cref{thm:KL_minimization}} \label{sec:proof:KL_minimization}

We start with two lemmas that we are going to use in our proof of 
\Cref{thm:KL_minimization}.

\begin{lemma}[Theorem 46 of \cite{SBGK18}] \label{lem:coefficient_bound_unit_cube}
    Let $p$ be a polynomial with real coefficients on $d$ variables with some
  degree $k$ such that $p \in L_\infty([0,1]^d, B)$.  Then, the magnitude of
  any coefficient of $p$ is at most $B (2k)^{3 k}$ and the sum of magnitudes
  of all coefficients of $p$ is at most $B \min((2 (d + k))^{3 k}, 2^{O(d k)})$.
\end{lemma}
We note that in \cite{SBGK18} the $(2 (d + k))^{3 k}$ upper bound is given,
the other follows easily from the single dimensional bound $2^{O(k)}$.

\begin{lemma}[\cite{Ren92b,Ren92a}] \label{lem:renegar}
    Let $p_i:\R^d \mapsto \R, i\in [m]$ be $m$ polynomials over the reals  each
  of degree at most $k$. Let
  $K = \{\bx \in \R^d:p_i(\bx)\geq 0, \text{ for all } i\in[m]\}$. If the
  coefficients of the $p_i$’s are rational numbers with bit complexity at most
  $L$, there is an algorithm that runs in time $\poly(L,(mk)^d)$ and decides if
  $K$ is empty or not. Furthermore, if $K$ is non-empty,the algorithm runs in
  time $\poly(L,(mk)^d,\log(1/\delta))$ and outputs a point in $K$ up to an
  $L_2$ error $\delta$.
\end{lemma}

\subsubsection{Objective Function of MLE}

Now we define our objective, which is the Kullback-Leibler divergence
between $f$ and the candidate distribution, or equivalently the
maximum-likelihood objective.
\begin{align}
  \label{eq:kullback_leibler_objective}
  L(\bv) &= \dkl(\D(f, S)  \| \D(\bv, S))
  \\
         &=\int \D(f, S;\bx) \log \D(f, S; \bx) \d \vec x
  -\int \D(f, S; \bx) \vec v^T \vec m_k(\vec x) \d \vec x
  + \log \int_S e^{\vec v^T \vec m_k(\bx)} \d \vec x
  \nonumber
\end{align}
The gradient of $L(\bv)$ with respect to $\vec v$ is
  \begin{align}
    \label{eq:one_dimensional_gradient}
    \nabla_{\vec v} L(\vec v)
    &= - \int \D(f,S;\bx) \vec m_k(\vec x) \d \vec x
    + \frac{\int_S \vec m_k(\bx)
    e^{\vec a^T \vec m_k(\vec x)}\d \vec x}
    {\int_S e^{\vec v^T \vec m_k(\vec x)}\d \vec x}
    \nonumber
    \\
    &=
    \E_{x \sim \D(\vec v, S)}[\vec m_k(\vec x)]
    - \E_{x \sim \D(f,S)}[\vec m_k(\vec x)]
  \end{align}
  The Hessian of $L(\vec a)$ with respect to $\vec v$ is
  \begin{align}
    \label{eq:one_dimensional_hessian}
    \nabla^2_{\vec v} L(\vec v)
    &=
    \frac{\int_S \vec m_k(\vec x) \vec m_k^T(\vec x)
    e^{\vec v^T \vec m_k(\vec x)} \d \vec x} {\int_S e^{\vec v^T \vec m_k(\vec x)}}
    -
    \frac{\int_S \vec m_k(\vec x) \vec m_k^T(\vec x) e^{\vec v^T \vec m_k(x)} \d \vec x}
    {\left(\int_S e^{\vec v^T \vec m_k(\vec x)} \right)^2}
    \nonumber
    \\
    &=
    \E_{x \sim \D(\vec v, S)}[\vec m_k(\vec x) \vec m_k^T(\vec x)]
    -
  \E_{x \sim \D(\vec v, S)}[\vec m_k(\vec x)]
  \E_{x \sim \D(\vec v, S)} [\vec m_k^T(\vec x)].
  \end{align}
  We observe that the Hessian is positive semi-definite since it is the
  covariance of the vector $\vec m_k(\vec x)$.  Therefore, we verify
  that $L(\bv)$ is convex as a function of $\bv$.

\subsubsection{Convergence of PSGD}

Now, we prove that using \Cref{alg:projectedSGD} we can efficiently
estimate the parameters of a polynomial whose density well approximates the
unknown density $\D(f, [0,1]^d)$ in the whole unit cube.

  We want to optimize the function $L(\bv)$ of 
  \Cref{eq:kullback_leibler_objective} constrained
  in the convex set
  $$D = \left\{
    \vec v \in \R^m~:~ \snorm{\infty}{\1_K \vec v^T \vec m_k(\bx) }
    \leq C
 \right\}.$$
 To be able to perform SGD we need to have unbiased estimates of the gradients.
 In particular, from the expression of the gradient (see
 \Cref{eq:one_dimensional_gradient}) we have that in order to have
 unbiased estimates we need to generate a sample from the distribution $\D(\vec
 v, S)$.  We first observe that the initialization $\vec v^{(0)} \in D$.
Using rejection sampling we can generate with probability at least
 $1-\delta$ a sample uniformly distributed on $S$ after
 $\log(1/\delta)/\alpha$ draws from the uniform distribution on $K$.  Using
 the samples distributed uniformly over $S$ we can use again rejection
 sampling to create a sample from $\D(\bv, S)$ using as base density the
 uniform over $S$.  Since $e^{\vec v^T \vec m_k(\bx)} \leq e^{C}$, the
 acceptance probability is $e^{-C}$.  We need to generate $e^{C}
 \log(1/\delta)$ samples from the uniform on $S$ in order to generate one
 sample from $\D(\vec v, S)$.  Overall, the total samples from the uniform on
 $K$ in order to generate a sample from $\D(\vec v, S)$ with probability
 $1-\delta$ is $O(e^{C} C \log(1/\delta))$.  To generate an unbiased
 estimate of the gradient we can simply draw samples $\bx_t \sim \D(\vec v,
 S), \by_t \sim \D(f, S)$ and then take their difference, i.e.  $\vec g^{(t)}
 = \vec m_k(\bx^{(t)}) - \vec m_k(\by^{(t)})$.  We have $\snorm{2}{\vec
 g^{(t)}}^2 \leq  2 \binom{d+k}{k}$ for any $\bx \in K$.  Moreover, we need a
 bound on the $L_2$ diameter of $D$.  From
 \Cref{lem:coefficient_bound_unit_cube} we have that since $\vec v^T
 \vec m_k(\bx) \in L_\infty(K, C)$ we get that $\snorm{2}{\vec v} \leq
 \snorm{1}{\vec v} \leq C (2(d+k))^k$.  Now, we have all the ingredients to
 use \Cref{lem:sgd_convex}, and obtain that after
 $$
 T = \frac{C^2 2^{O(d k)} \cdot \binom{d+k}{k}}{\eps^2}
 = \frac{C^2 2^{O(dk)}}{\eps^2}
 $$
 rounds, we have a vector $\vec v^{(T)}$ with optimality gap $\eps$.

 We next describe an efficient way to project to the convex set $D$.  The
 projection to $D$ is defined as $\argmin_{\vec u \in D} \snorm{2}{\vec u -
 \vec v}^2$.  We can use the Ellipsoid algorithm (see for example \cite{GLS93})
 to optimize the above convex objective as long as we can implement a
 separation oracle for the set $D$.  The set $D$ has an infinite number of
 linear constraints (one constraint for each $\bx \in K$ but we can still use
 Renegar's algorithm to find a violated constraint for a point $\bv \notin D$.
 Specifically, given a guess $\vec v$ we set up the following system of
 polynomial inequalities,
 \begin{align*}
 &\vec v^T \vec m_k(\bx) \geq C \\
 &0 \leq x_i \leq 1 \text{ for all } i \in [d]\,,
 \end{align*}
 where $\bx$ is the variable. Using \Cref{lem:renegar} we can decide if
 the above system is infeasible or find $\bx$ that satisfies $\vec v^T \vec
 m_k(\bx) \geq C$ in time $\poly(((d+1)k)^{d})$, where we suppress the
 dependence on the accuracy and bit complexity parameters.\footnote{Since the
   dependence of Renegar's algorithm is polynomial in the bit size of the
   coefficients and the accuracy of the solution it is straightforward to do
   the analysis of our algorithm assuming finite precision rational numbers
 instead of reals.}   If Renegar's algorithm returns such an $\bx$ we have a
 violated constraint of $D$.  Since $D$ bounds the absolute value of $\vec
 v^T \vec m_k(\bx)$ we need to run Renegar's algorithm also for the system
 $\{\bx : \vec v^T\vec m_k(\bx) \leq -C, \bx \in K\}$.  The overall runtime
 of our separation oracle is  $\poly(((d+1)k)^{d})$ and thus the runtime of
 Ellipsoid to implement the projection step is also of the same order.
 Combining the runtime for sampling, the projection, and the total number of
 rounds we obtain that the total runtime of our algorithm is
 $ 2^{O(dk + C)}/(\alpha \eps^2)$.
 \qed

\begin{algorithm}[H]
  \caption{Projected Stochastic Gradient Descent given access to samples from $\D(f, S)$.}
  \label{alg:projectedSGD}
  \begin{algorithmic}[1]
    \Procedure{Sgd}{$T, \eta, C$}\Comment{$T$: number of steps,
    $\eta$: step size, $C$: projection parameter.}
    \State $\vec v^{(0)} \gets \vec 0$
    \State Let $D = \{\vec v: \max_{\bx \in K} |\vec v^T \vec m_k(\bx)| \leq C \}$
    \For{$t = 1, \dots, T$}
    \State Draw sample $\vec x^{(t)} \sim \D(\vec v^{(t-1)}, S)$
    and $\vec y^{(t)} \sim \D(f, S)$
    \State $\vec g^{(t)} \gets \vec m_k(\bx^{(t)}) - \vec m_k(\by^{(t)}) $
    \State $\vec v^{(t)} \gets
    \mathrm{proj}_D(\vec v^{(t-1)} - \eta \vec g^{(t)} ) $
    \EndFor\label{euclidendwhile}
    \State \textbf{return}
  \EndProcedure
  \end{algorithmic}
  \end{algorithm}
 
\end{document}